\documentclass{amsart}
\usepackage[utf8]{inputenc}
\usepackage{a4wide}
\usepackage{amssymb,exscale,xpatch,mathrsfs,nicefrac,dutchcal,tikz,pgfplots, booktabs, array, float}
\usepackage[hidelinks]{hyperref}
\usepackage[pagewise]{lineno}
\usepackage{comment}

\usepackage{bbm}

\pgfplotsset{width=6cm}

\numberwithin{equation}{section}
\theoremstyle{plain}
\newtheorem{theorem}{Theorem}[section]
\newtheorem{theorem*}{Theorem}
\newtheorem{lemma}[theorem]{Lemma}
\newtheorem{corollary}[theorem]{Corollary}

\newtheorem{proposition}[theorem]{Proposition}

\theoremstyle{definition}

\newtheorem{example}[theorem]{Example}

\theoremstyle{remark}
\newtheorem{remark}[theorem]{Remark}

\theoremstyle{plain}
\newtheorem{assumptionalt}{Assumption}
\newenvironment{assumptionp}[1]{
  \renewcommand\theassumptionalt{#1}
  \assumptionalt
}{\endassumptionalt}

\renewcommand{\arraystretch}{1.4} 

\newcommand{\dst}{\displaystyle}

\newcommand{\N}{\mathbb{N}}
\newcommand{\Z}{\mathbb{Z}}

\newcommand{\R}{\mathbb{R}}

\newcommand{\C}{\mathbb{C}}

\newcommand{\T}{\mathbb{T}}

\newcommand{\ui}{\mathrm{i}}

\newcommand{\eps}{\varepsilon}
\newcommand{\ffi}{\varphi}

\renewcommand{\d}{\,\mathrm{d}}
\newcommand {\sfrac}[2] { {\,{}^{#1}\!\!/\!{}_{#2}\,}} 
\newcommand {\onehalf} {\sfrac{1}{2}}

\let\Re\relax\DeclareMathOperator{\Re}{Re}
\let\Im\relax\DeclareMathOperator{\Im}{Im}

\newcounter{aufzi}
\newenvironment{aufzi}{\begin{list}{ {\upshape\alph{aufzi})}}{
        \usecounter{aufzi}
        \topsep1ex
        \parsep0cm
        \itemsep1ex
        \leftmargin0.8cm
        \labelwidth0.5cm
        \labelsep0.3cm
}}
{\end{list}}

\newcounter{aufzii}

\newcounter{aufziii}

\colorlet{yunlei}{purple}

\def\cc{{\mathcal C}}

\def\ff{{\mathcal F}}

\def\mm{{\mathcal M}}

\def\pp{{\mathcal P}}

\newcommand{\norm}[1]{{\left\|{#1}\right\|}}
\newcommand{\ent}[1]{{\left[{#1}\right]}}
\newcommand{\abs}[1]{{\left|{#1}\right|}}
\newcommand{\scal}[1]{{\left\langle{#1}\right\rangle}}

\newcommand{\proofskip}{\par\smallskip\noindent}

\begin{document}
\raggedbottom
\title[Curved Ingham Inequalities]{Curved Ingham Inequalities for Dispersive\\Trigonometric Sums}

\author[Bernhard H. Haak]{Bernhard H. Haak}
\address{Univ. Bordeaux, CNRS, Bordeaux INP, IMB, UMR 5251, F-33400 Talence, France}
\email{bernhard.haak@math.u-bordeaux.fr}
\author[Philippe Jaming]{Philippe Jaming}
\address{Univ. Bordeaux, CNRS, Bordeaux INP, IMB, UMR 5251, F-33400 Talence, France}
 \email{philippe.jaming@math.u-bordeaux.fr} 
\author{Ming Wang}
\address{School of Mathematics and Statistics, HNP-LAMA, Central South University, Hunan, Changsha,
410083, PR China}
\email{m.wang@csu.edu.cn}
\author[Yunlei Wang]{Yunlei Wang}
\address{Department of Mathematics, Louisiana State University, Baton Rouge, LA 70803, USA}
\email{ywang30@lsu.edu}

\thanks{Corresponding author: Philippe Jaming, philippe.jaming@math.u-bordeaux.fr}


\date{\today}

\begin{abstract}
We study $L^2$ restrictions of dispersive trigonometric sums
$$
F(t,x)=\sum_{n\in\Z} c_n e^{2\pi i(|n|^s t+nx)}
$$
to curves in the $(t,x)$-plane. Our main result is a curved analogue of
Ingham's inequality: for a class of trajectories satisfying quantitative
curvature assumptions, we obtain upper and lower $L^2$ estimates for the
restriction of $F$ in terms of the $\ell^2$-norm of its coefficients. 

We then investigate two situations in which this obstruction disappears.
First, for measures with polynomial Fourier decay, sufficiently high
frequencies form a Riesz sequence; in particular, along smooth curves with
non-vanishing curvature this yields strong Ingham inequalities for sufficiently high frequencies when $s>2$.
Second, for every fixed time interval, the full system satisfies such a
lower estimate when the dispersion exponent $s$ is sufficiently large.

We also study finite dispersive trigonometric sums and prove rigidity
results for their zero sets along analytic and meromorphic curves.
Finally, we apply these harmonic-analysis results to fractional
Schrödinger equations on the torus, including equations with bounded
potentials.
\end{abstract}

\maketitle
\tableofcontents

\allowdisplaybreaks


\section{Introduction}

\subsection{Motivation}
In this paper, we study the restriction of dispersive trigonometric sums
\begin{equation}
\label{eq:disptrigsum}
F(t,x)=\sum_{n\in\Z}c_n e^{2\pi i(|n|^s t+nx)},\qquad s>1,
\end{equation}
to curves in the $(t,x)$-plane, and address two natural questions. For finite
sums, we ask which curves can be contained in the zero set of a non-trivial
dispersive trigonometric polynomial. For infinite sums with
$(c_n)\in\ell^2(\mathbb Z)$, we ask whether restriction to a curve defines a
bounded $L^2$ trace operator and, if so, whether the norm of this trace controls
the $\ell^2$-norm of the coefficients. The aim of this paper is to provide
answers to both questions for a broad class of curved trajectories and to
exhibit the close connection between them.

A natural starting point is the case of affine curves. After an affine
reparametrization, the restriction of a dispersive trigonometric sum to such
a curve takes the form
$$
\sum_{n\in\Z} d_n e^{2\pi \ui(n+a|n|^s)t},
$$
where $|d_n|=|c_n|$. 
Unless $a$ belongs to an exceptional countable set, the frequencies $n+a|n|^s$ are distinct and can therefore be
rearranged into a strictly increasing sequence $(\lambda_k)_{k\in\Z}$.
Whenever this sequence is uniformly separated, the classical Ingham inequality
applies\,:

\begin{theorem}[Ingham's Inequality \cite{Ingham}]
  Let $\gamma>0$ and $(\lambda_n)_{n\in\Z}$ be a sequence of real
  numbers with $\lambda_{n+1}-\lambda_n\geq \gamma$.  Then, for every
  $T>0$, there exists $C(T)$ such that, for every
  $(c_n)_{n\in\Z}\in\ell^2$,
\begin{equation}
  \label{eq:inghamd}
  \frac{1}{T}\int_0^T \Bigl| \sum_{n\in\Z}c_ne^{2\ui\pi\lambda_n t} \Bigr|^2\d t\leq C(T)\sum_{n\in\Z}|c_n|^2.
\end{equation}
Further, if $T>\dfrac{1}{\gamma}$, there exists $\tilde C(T)$ such that, for
every $(c_n)_{n\in\Z}\in\ell^2$,
\begin{equation}
  \label{eq:inghami}
  \frac{1}{T}\int_0^T\Bigl|\sum_{n\in\Z}c_ne^{2\ui\pi\lambda_n t} \Bigr|^2\d t\geq \tilde C(T)\sum_{n\in\Z}|c_n|^2.
\end{equation}
\end{theorem}

Hence, \eqref{eq:inghamd} shows that dispersive trigonometric sums have traces along segments, regardless of their length, while \eqref{eq:inghami} shows that if the trace is $0$
then the dispersive trigonometric sum $F$ in \eqref{eq:disptrigsum} vanishes identically, provided the segment is long enough.

In this theorem, the constants $C(T)$ and $\tilde C(T)$ are explicit, and the
condition $T>\dfrac{1}{\gamma}$ is sharp in general. However, in the case under consideration,
the gaps $\lambda_{k+1}-\lambda_k$ are not only bounded from below, but $\lambda_{k+1}-\lambda_k\to+\infty$
when $k\to\pm\infty$. When the gaps tend to infinity, the minimal-time restriction disappears. This leads to the following strong form of Ingham's inequality\,:

\begin{theorem}[Strong Ingham Inequality]
  Let $(\lambda_n)_{n\in\Z}$ be a strictly increasing sequence of real numbers 
  such that $\lambda_{k+1}-\lambda_k\to+\infty$ when $k\to\pm\infty$.
Then, for every $T>0$, there exists $\tilde C(T)$ such that, for
every $(c_n)_{n\in\Z}\in\ell^2$, \eqref{eq:inghami} holds.
\end{theorem}

This follows from a result of Kahane \cite{Kahane}; Haraux \cite{Haraux}
later gave a constructive proof yielding quantitative estimates on the constants.

The upper bound shows that restriction to the segment defines a bounded
operator from $\ell^2(\Z)$ into $L^2(0,T)$, whereas the lower bound
shows that this restriction is bounded from below. In particular, a vanishing
trace forces all coefficients to vanish. Thus, for affine curves, both the
trace and the quantitative uniqueness questions are answered by classical
non-harmonic Fourier analysis.

For further details, we refer to the work of the second author with V. Komornik \cite{JamingKomornik}
and the work of the two last authors \cite{WangWang}.

\subsection{Main results}

We now replace affine segments by genuinely curved trajectories
\[
\Gamma_p=\{(t,p(t)):t\in[0,T]\}.
\]
The model case is the power curve $p(t)=t^\alpha$, $\alpha>1$. The essential
difference with the affine case is that
\[
t\longmapsto e^{2\pi i(np(t)+|n|^s t)}
\]
no longer has a linear phase. Hence classical non-harmonic Fourier theory does
not apply directly, and the separation of frequencies has to be replaced by
estimates for oscillatory integrals involving both the dispersion relation and
the geometry of the curve.

\proofskip In order to state our results, let us introduce the family of curves that we will consider.

\medskip

\begin{assumptionp}{${\rm \mathbf{(H_\alpha)}}$}\label{ass:H-alpha}
  Let $p\,:[0,\infty)\to \R$ be a function in
  $\cc^1([0,\infty[)\cap \cc^2(]0,\infty[)$ and $\alpha>1$.  We will say
  that $p$ satisfies Assumption~\ref{ass:H-alpha} with parameters
  $0<c_1<c_2$ and $c_3>0$ if, for every $t\geq 0$,
    \begin{equation}\label{ass:alpha-1}
        c_1 \, t^{\alpha-1} \le |p'(t)|\le c_2\, t^{\alpha-1},
\end{equation}
and, for every $t>0$,
\begin{equation}\label{ass:alpha-2}
    |p''(t)|\ge c_3 \, t^{\alpha-2}.
\end{equation}
\end{assumptionp}

The model example is $p(t)=a+bt^\alpha$, with $b\neq0$. Assumption
$(H_\alpha)$ also allows smooth perturbations of such power curves.

Our first main result is the following:

\begin{theorem}[Curved Ingham Inequality]
\label{thm:schroedinger}
Let $\alpha>1$, $s> \dfrac{3}{2}$. Let
$p\in \cc^1([0,\infty[)\cap \cc^2(]0,\infty[)$ satisfying
Assumption~\ref{ass:H-alpha}. Then, for every $T>0$ there exists a
constant $C(T)$ such that, for every sequence
$(c_n)_{n\in\Z}\in\ell^2$, we have
\begin{equation}\label{eq:obs-curve-schroedinger-upper}
\int_0^T\abs{\sum_{n\in\Z}c_ne^{2\ui\pi (np(t)+|n|^st)}}^2\,\mathrm{d} t
\leq C(T) \sum_{n\in\Z}|c_n|^2.
\end{equation}
If, in addition, $s>\max \left(\dfrac{3}{2}, 2-\dfrac{1}{\alpha} \right)$, then
there exist constants $T(p,s)>0$ and $\tilde C(p,s)>0$ such that
for every $T> T(p,s)$ and every sequence $(c_n)_{n\in\Z}\in\ell^2$, we
have
\begin{equation}\label{eq:obs-curve-schroedinger}
\tilde C(p,s) \sum_{n\in\Z}|c_n|^2\leq 
\frac{1}{T}\int_0^T\abs{\sum_{n\in\Z}c_ne^{2\ui\pi (np(t)+|n|^st)}}^2\,\mathrm{d} t.
\end{equation}
\end{theorem}

The two estimates have precisely the interpretations suggested above.
Our proof of the lower estimate requires sufficiently large time. Lemma~\ref{prop:minimal} shows that a given positive normalized lower constant cannot hold uniformly on arbitrarily short intervals. This does not rule out a positive constant depending on $T$ for every $T>0$.

The proof is based on oscillatory integral estimates. After expanding the
square, one is led to integrals
\[
I_{n,m}(T)=
\int_0^T
e^{2\pi \ui((n-m)p(t)+(|n|^s-|m|^s)t)}\,\mathrm{d}t.
\]
We separate the pairs $(n,m)$ according to whether the corresponding phase has
a stationary point, combine Van der Corput estimates with the assumptions on
$p$, and finally sum the resulting off-diagonal bounds by a Schur-type
argument.

\smallskip

We next establish lower estimates on every fixed time interval in two regimes. The first is a high-frequency result: after
discarding sufficiently many low modes, a strong Ingham inequality holds on
arbitrarily short curved arcs. The second concerns strong dispersion: for any
fixed $T>0$, the full system satisfies a lower estimate when $s$ is sufficiently
large.

When only sufficiently high frequencies are retained, an analogue of the
lower estimate in \eqref{eq:obs-curve-schroedinger} holds for every $T>0$.
This will follow from a more general result: consider a planar measure
$\mu$ that has polynomial Fourier decay.  We will show that under
suitable conditions on $s,N$ and the polynomial decay of
$\widehat{\mu}$, the $L^2(\mu)$-norm of a series of the form
$\dst\sum_{|n|\geq N}c_ne^{2\ui\pi(nx+|n|^st)}$ is equivalent to the
$\ell^2$-norm of the coefficients $(c_n)_{|n|\geq N}$, {\it i.e.}  that
the system of functions $(e^{2\ui\pi(nx+|n|^st)})_{|n|\geq N}$ is a
Riesz sequence in $L^2(\mu)$, when $N$ is large enough, see
Theorem~\ref{thm:highfreqmeas}.  A corollary of this result can be formulated as follows.

\begin{theorem}\label{thm:highfreq}
  Let $s>2$ and $T>0$. Let $\Gamma$ be a compact smooth curve in $[0,T]\times\T$ with
  non-zero curvature and $\sigma$ be the normalized arc-length measure
  on $\Gamma$.  Then there exists $N(s,\Gamma)$ such that, if
  $N\geq N(s,\Gamma)$, for every $(c_n)_{|n|\geq N}\in\ell^2$,
\begin{equation}
    \label{eq:thmhighfreqNbigcurve}
    \dfrac{1}{2}\sum_{|n|\ge N}|c_n|^2\leq  \int_{\Gamma} \abs{\sum_{|n|\geq N}c_n e^{2\ui\pi(|n|^st+nx)}}^2
    \,\mathrm{d}\sigma(t,x) 
    \leq \dfrac{3}{2}\sum_{|n|\ge N}|c_n|^2.
\end{equation}
\end{theorem}

\proofskip As a consequence, we obtain the following:

\begin{corollary}[High-Frequency Curved Ingham Inequality]
\label{cor:highfreqintro}
Let $s>2$, $\alpha\geq 2$ and
$p\in \cc^1([0,\infty[)\cap \cc^2(]0,\infty[)$ satisfying
Assumption~\ref{ass:H-alpha}.  Then, for every $T>0$ there exist constants
$N(T,p)\in\N$, $C(T,p)$ and $\tilde C(T,p)$ such that, for every
$N\geq N(T,p)$ and every sequence $(c_n)_{|n|\geq N}\in\ell^2$, we
have
\begin{equation}\label{eq:obs-curve-schroedinger-highfreq}
\tilde C(T,p)\sum_{|n|\geq N}|c_n|^2
\leq\int_0^T\abs{\sum_{|n|\geq N}c_ne^{2\ui\pi (np(t)+|n|^st)}}^2\,\mathrm{d} t
\leq C(T,p) \sum_{|n|\geq N}|c_n|^2.
\end{equation}

Moreover, there exists $s(T)$ such that, if $s>s(T)$, the previous
inequality is true for all $N\geq 0$, that is, for every sequence
$(c_n)_{n\in\Z}\in\ell^2$, we have
\begin{equation}\label{eq:obs-curve-schroedinger-highfreq-highdisp}
\tilde C(T,p)\sum_{n\in\Z}|c_n|^2
\leq\int_0^T\abs{\sum_{n\in\Z}c_ne^{2\ui\pi (np(t)+|n|^st)}}^2\d t.
\end{equation}
\end{corollary}

\proofskip 
The second part of the corollary requires additional work
which is detailed in Section~\ref{Sec:highdisp}.

\bigskip 

We finally return to the first question raised at the beginning of the introduction,
namely the zero sets of finite dispersive trigonometric sums, that is, functions of the form
\begin{equation}  \label{eq:4.1-copy}
  F(t,x)=\sum_{n=-N}^N  c_n \,  e^{2\ui \pi \bigl(|n|^s t + \lambda_n x\bigr)}.
\end{equation}
When $N$ is fixed, these functions form a finite dimensional vector
space.  Now if a curve
$\Gamma=\bigl\{\bigl(t,\gamma(t)\bigr),\ t\in[0,T]\bigr\}$ is such
that the only function $F$ of the form \eqref{eq:4.1-copy} that
vanishes on $\Gamma$ is $F=0$ then
\[
(c_n)_{n=-N,\ldots,N}\to \int_0^T\abs{\sum_{n=-N}^N  c_ne^{2\ui\pi \bigl(|n|^s t + \lambda_n \gamma(t)\bigr)}}^2\,\mathrm{d}t
\]
is the square of a norm on $\cc^{2N+1}$. Therefore, for every strictly positive continuous weight $w$ on $[0,T]$, there exists some constant
$C=C(N,\gamma,w)\ge1$ such that, for all $(c_n)_{|n|\le N}$,
\begin{equation}  \label{eq:weighted-norm-equivalence}
\frac{1}{C}\sum_{n=-N}^N|c_n|^2
\le
\int_0^T |F(t,\gamma(t))|^2\,w(t)\,\mathrm dt
\le
C\sum_{n=-N}^N|c_n|^2.
\end{equation}

%

We will show the following rigidity result:

\begin{theorem}\label{thm:lowfreq-intro}
  Let $s,T>0$ and let $N\ge0$.  Let $(\lambda_n)_{n=-N}^N\subset\R$ be
  pairwise distinct frequencies and define $F$ as in
  \eqref{eq:4.1-copy}. Let $\gamma:[0,T]\to\R$ be continuous.
\begin{enumerate}
\item If
\[
  F(t,\gamma(t))\equiv0\quad\text{and} \quad  \partial_x F(t,\gamma(t))\neq 0 \quad \text{for all } t\in(0,T),
\]
 then $\gamma$ extends holomorphically to a
complex neighborhood of $(0,T)$.

\item Assume that $\gamma$ extends holomorphically and satisfies at
  least one of the following conditions:
  \begin{itemize}
  \item $\gamma$ is entire and not affine;
  \item $\gamma$ extends meromorphically to a connected complex neighborhood of
    $[0,T]$ and has at least one pole.
\end{itemize}
Then
\[
  \{t\in(0,T):F(t,\gamma(t))=0\}\ \text{has an accumulation point}
  \quad\Longrightarrow\quad F\equiv0.
\]
\end{enumerate}
\end{theorem}

Zero sets of trigonometric polynomials are well studied objects, see
e.g.  \cite{BourgainRudnik,BurqGermainSorellaZhu,DL23,GermainMoyanoZhu} and
references therein. Most results deal with trigonometric sums that are
eigenfunctions of the Laplacian and are thus sums of the form
\[
   \sum_{m^2+|n|^2=N^2}c_{m,n} e^{2\ui\pi\bigl(m t+ \langle n,x \rangle \bigr)}.
\]
Thus, the Fourier support of those functions is a sphere
$\{(m,n)\in\Z\times \Z^d\,:\ m^2+|n|^2=N^2 \}$.  In our case however,
the Fourier support is not a sphere, but a dispersive curve $\{(|n|^s,n),\ n\in\Z\}$
(a parabola when $s=2$).

\smallskip

Finally, we extend Theorem~\ref{thm:schroedinger} from dispersive Fourier series to solutions
of fractional Schrödinger equations with bounded potentials $V\in L^\infty(\mathbb T)$
\begin{equation*}
\begin{cases}
    i\partial_tu=(-|\partial_x|^s+V)u, \quad x\in \T,t\in \R,\\
    u(0,x)=u_0\in L^2(\T)   .
\end{cases}
\end{equation*}
We show that every solution admits an $L^2$ trace along curves satisfying
Assumption~\ref{ass:H-alpha}. Moreover, for sufficiently small potentials and sufficiently
large times, this trace controls the $L^2$-norm of the initial datum.
Whether such a lower bound holds for arbitrary bounded potentials and arbitrarily
small times remains open.

We give a short overview of the literature on control problems
for the Shr\"odinger equation in Appendix \ref{sec:bib-table}.

\medskip 

The remainder of the paper is organized as follows: We prove
Theorem~\ref{thm:schroedinger} in Section~\ref{Sec:thm:schroedinger}.
This requires several intermediate steps: a Van der Corput type
estimate in Section~\ref{Sec:vdc}, and two simple lemmas about
sequences and series in Sections~\ref{Sec:seq} and \ref{Sec:ser}. The
proof itself is given in Section~\ref{Sec:pf:schroedinger}.  We then
prove Theorem~\ref{thm:highfreq} and
Corollary~\ref{cor:highfreqintro} in Section~\ref{Sec:Th2}. In
Section~\ref{Sec:pfTh2} we prove a more general result and show
how it implies Theorem~\ref{thm:highfreq}. We conclude the section with the proof of the
second part of Corollary~\ref{cor:highfreqintro} in
Section~\ref{Sec:highdisp}.  We then devote Section~\ref{Sec:lowfreq}
to the proof of Theorem~\ref{thm:lowfreq-intro} and conclude in
Section~\ref{Sec:schroedinger-potential} with the case of
Schrödinger equations with potentials.

\section{Curved Ingham Inequality}\label{Sec:thm:schroedinger}

\noindent Our strategy to prove the Curved Ingham Inequality is as
follows.  We write
\[
\int_0^T\abs{\sum_{n\in\Z}c_ne^{2\ui\pi (np(t)+|n|^st)}}^2\d t=
\sum_{n,m\in\Z}c_n\overline{c_m}\, J_{n,m}
\]
with
\begin{equation}
J_{n,m}:=J_{n,m}(T):=\int_0^Te^{2\ui\pi \bigl((n-m)p(t)+(|n|^s-|m|^s)t\bigr)}\d t.\label{eq:without-abs}
\end{equation}
Then, using stationary phase methods, we prove that, under good
conditions on $s,\alpha$ and if $T$ is large enough,  then $|J_{n,m}|$ has
sufficient off-diagonal decay to establish that
\[
\int_0^T\abs{\sum_{n\in\Z}c_ne^{2\ui\pi \bigl(np(t)+|n|^st\bigr)}}^2\d t\simeq\sum_{n\in\Z}|c_n|^2.
\]

\noindent The proof of Theorem~\ref{thm:schroedinger} is preceded by three
technical lemmata.

 \subsection{A first technical result}\label{Sec:seq}
Let us start with the following simple lemma that we will need later.

\begin{lemma}\label{lem:bounded}
  Let $\gamma>0$ and $s>1$. Then  
 the supremum
    \begin{equation*}
      M(\gamma,s)=\sup\left\{\frac{|n-m|}{\bigl||n|^s-|m|^s\bigr|^\gamma}\,:\ m,n\in\Z,\ |m|\neq|n|\right\}
    \end{equation*}
    is finite if and only if $(s-1)\gamma \ge 1$. In this case,  $M(\gamma,s) \le  4$.
\end{lemma}

\begin{proof} To show necessity, we simply
    choose $n=-m-1$ and $m\ge 1$. The Mean Value Theorem then implies
    \begin{equation*}
        \frac{|n-m|}{\bigl||n|^s-|m|^s\bigr|^\gamma}= \frac{2m+1}{\bigl|(m+1)^s-m^s\bigr|^\gamma} \sim\frac{2}{s^\gamma}|m|^{1-(s-1)\gamma}.
    \end{equation*}
    Letting $m\to+\infty$, this shows that the condition
    $(s-1)\gamma\geq 1$ is necessary to have a uniform upper bound.
    For the converse, we assume that $(s-1)\gamma\ge 1$. Let
  \[
      a=\min\{|m|,|n|\}\ge 0
    \quad\text{and}\quad 
    k=\bigl||n|-|m|\bigr|\geq 1
  \]
  since $|n|\not=|m|$. By construction $\max(|m|,|n|)=a+k$. This implies that
  $|n-m|\le 2\max(|m|,|n|)\leq 4\max\{a,k\}$. Write,
\[
  (a+k)^s-a^s = s\int_0^1 (a+tk)^{s-1}\,k\,\mbox{d}t.
\]
Using that $(a+tk)^{s-1}\geq a^{s-1}$ inside the integral we obtain 
\[
(a+k)^s-a^s\geq sa^{s-1}k
\]
while, using $(a+tk)^{s-1} \ge (tk)^{s-1}$ inside the integral gives
\[
(a+k)^s-a^s \;\ge\; k^s.
\]
Consequently, the quotient
\[
  f(n,m) := \frac{|n-m|}{\bigl||n|^s-|m|^s\bigr|^\gamma}
  \;\le\; \min\left(\frac{4\max\{a,k\}}{s^\gamma a^{(s-1)\gamma}k^\gamma},\frac{4\max\{a,k\}}{k^{s\gamma}}\right).
\]
We now split into two cases.

\proofskip
\emph{Case~1: $a\ge k\geq 1$.} We use the first bound 
\[
  f(n,m)\le \frac{4}{s^\gamma}\,\frac{1}{a^{(s-1)\gamma-1}k^\gamma}\le \frac{4}{s^\gamma k^\gamma}
\]
since $(s-1)\gamma\ge1$ and $a\geq 1$. Further, the right-hand side is uniformly bounded by $4/s^\gamma$.

\proofskip
\emph{Case~2: $k\ge a$.} Then the second bound gives
\[
  f(n,m)\le\frac{4}{k^{s\gamma-1}}\leq 4
\]
since $s\gamma\geq 1+\gamma>1$.

\proofskip
We have obtained a uniform bound $M(\gamma,s) \le  4\max(1,s^{-\gamma})=4$.
\end{proof}

\begin{remark}\label{rem:yunlei}
It turns out that
    \begin{equation*}
      \inf\left\{\frac{|n-m|}{\bigl||n|^s-|m|^s\bigr|^\gamma}\,:\ m,n\in\Z,\ |m|\neq |n|\right\}=0.
    \end{equation*}
    To see this, when $0<\gamma<1$, it is enough to consider $n=m+1$  and $m\to+\infty$,  while for $\gamma\ge 1$, one may consider $n=2m$
    and $m\to+\infty$.  The result is unfortunate as a positive lower
    bound would have lead to a simpler proof of part of
    Theorem~\ref{thm:schroedinger}.
\end{remark}

\subsection{A Van der Corput type estimate}\label{Sec:vdc}

\begin{proposition}\label{prop:vdc}
  Let $s>1$, $\alpha>1$ and
  $p\in \cc^1([0,\infty[)\cap \cc^2(]0,\infty[)$ satisfying Assumption~\ref{ass:H-alpha} with constants $c_1,c_2,c_3$
  and such that $p'>0$ for $t>0$. Let
  $T_0 := \left(\dfrac{3(\alpha-1)}{4\pi c_1}\right)^{\frac{1}{\alpha}}$.  For
  $T>0$ and $n \not= m$ define
\[
   I_{n,m}(T)=|J_{n,m}(T)|=\abs{\int_0^Te^{2\ui\pi((n-m)p(t)+(|n|^s-|m|^s)t)}\,\mathrm{d}t}.
 \]
We define a threshold $\tau := 2c_2 T^{\alpha-1}$ and consider the
following partition of $\Z^2 \setminus \{ (n,\pm n): n \in \Z\}$:
\begin{align*}
    S^{+}_{\mathrm{good}} :=&\,\Bigl\{ (n,m)\in\Z^2:\ |n|\neq |m|,\ 
       \tfrac{|n|^s-|m|^s}{n-m} >0 \Bigr\},\\
    S^{-}_{\mathrm{good}}:=&\,\Bigl\{ (n,m)\in\Z^2:\ |n|\neq |m|,\ 
       \tfrac{|n|^s-|m|^s}{n-m}\leq -\tau \Bigr\},\\
    S_{\mathrm{bad}}:=&\,
    \Bigl\{ (n,m)\in\Z^2:\ |n|\neq |m|,\ 
        -\tau <\tfrac{|n|^s-|m|^s}{n-m}< 0\Bigr\}.
\end{align*}
\begin{figure}
    \centering
    \includegraphics[scale=0.5]{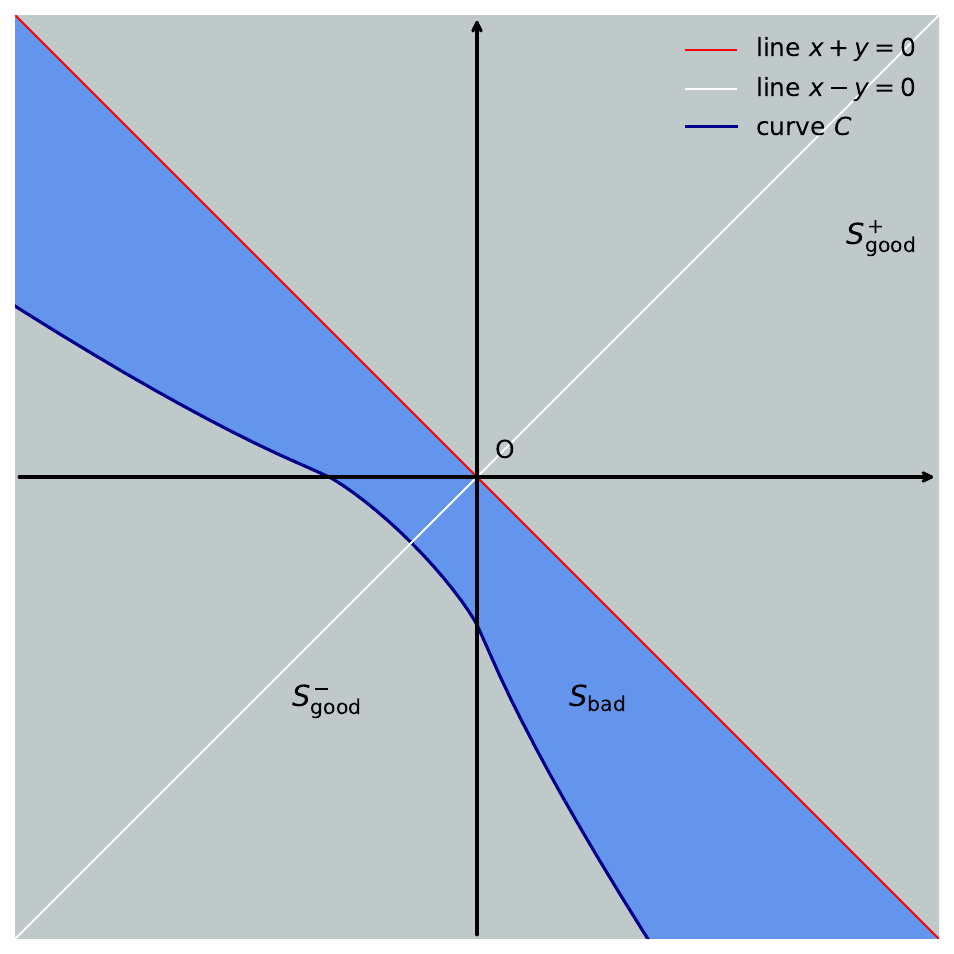}
    \caption{The sets $S^{+}_{\mathrm{good}}$ and $S^{-}_{\mathrm{good}}$
      in gray, and $S_{\mathrm{bad}}$ in blue, when $s=1.5$
      and $\tau = 4$.}
    \label{fig:S-sets}
  \end{figure}
  We also write
  $S_{\mathrm{good}} = S^{+}_{\mathrm{good}}\cup
  S^{-}_{\mathrm{good}}$. Then the following estimates hold:
\begin{enumerate}
\renewcommand{\theenumi}{\roman{enumi}}
\item For $(m,n)\in S_{\mathrm{good}}$, we have  $I_{m,n}(T)\lesssim ||n|^s-|m|^s|^{-1}$.
\item For each  $T\geq T_0$ and $n\not=0$, we have  $I_{n,-n}(T)\lesssim |n|^{-\frac{1}{\alpha}}$.
\item There is a $T_1\geq 0$ depending on $p$ and $\alpha$ as well as
  a parameter $\eta \le 1$ according to Table~\ref{table:eta} such
  that, for each $T>T_1$ and $(n,m)\in S_{\mathrm{bad}}$ we have
  \[
    I_{n,m}(T)\lesssim
    T^{\frac{1-\eta}{2}}|n-m|^{-\frac{\eta}{2(\alpha-1)}}\abs{|n|^s-|m|^s}^{-\frac{\alpha-1-\eta}{2(\alpha-1)}}
  \]
  with implied constants being independent of $T$ and $m,n$, but not
  of the chosen $\eta$, nor on $T_1$.
\begin{table}[ht]
\centering
\renewcommand{\arraystretch}{1.3} 
\setlength{\tabcolsep}{18pt}      
\begin{tabular}{l c}  
\toprule
\rm{Condition on $s,\alpha$} & $\eta$ \\
\midrule
$1<s<1+\frac{1}{\alpha}$ & $  -(\alpha-1) <  \eta <  \dfrac{s-1}{2-s}(\alpha-1)$ \\ 
$1+\frac{1}{\alpha}\le s$ & $  -(\alpha-1) < \eta \le 1$ \\
\bottomrule
\end{tabular}
\vspace{1em}
\caption{Admissible $\eta$'s.}
\label{table:eta}
\end{table}
%
\end{enumerate}
\end{proposition}

\begin{remark}
  The proof also shows that the estimate
  $I_{n,-n}(T)\lesssim |n|^{-\frac{1}{\alpha}}$ is also valid when $T>0$
  and $|n|\geq N_T:=\dfrac{3(\alpha-1)}{4\pi c_1 T^\alpha}$.
  It is also noteworthy that $T_1 \to 1^+$ as $\eta\to 0^+$.
\end{remark}

\begin{remark} 
If $p$ satisfies Assumption~\ref{ass:H-alpha}, $|p'(t)|>0$ for all $t>0$. In particular,
$p'$ has constant sign. So if $p'<0$, we can introduce $q=-p$ and notice
that
\begin{align*}
I_{m,n}^p(T):= & \; \abs{\int_0^Te^{2\ui\pi((n-m)p(t)+(|n|^s-|m|^s)t)}\,\mathrm{d}t}\\
= & \; \abs{\int_0^Te^{2\ui\pi(-(n-m)q(t)+(|n|^s-|m|^s)t)}\,\mathrm{d}t}\\
= & \; \abs{\int_0^Te^{2\ui\pi((-n-(-m))q(t)+(|-n|^s-|-m|^s)t)}\,\mathrm{d}t}
=I_{-m,-n}^q(T).
\end{align*}
We can thus obtain similar estimates in this case, after straight-forward
modification of the sets $S_{\mathrm{good}}^+, S_{\mathrm{good}}^-, S_{\mathrm{bad}}$.
\end{remark}

The proof of this proposition is based on the well-known Van der
Corput Lemma, that we recall for the convenience of the reader.  For
the proof, see for instance \cite[Corollary~2.9.3.(a)]{Grafakos}:

\begin{lemma}[Van der Corput]\label{lem:vdc}
  Let $-\infty<a<b<+\infty$,
  and $\phi$ a function of class $\cc^1$ on $[a,b]$.  Assume that
  $\phi'$ is monotonic and $|\phi'|\geq \xi>0$ on $[a,b]$ then, for
  every $\lambda\not=0$,
  \[
    \abs{\int_a^b e^{\ui \lambda
        \phi(t)}\,\mathrm{d}t}
    \leq \dfrac{3}{\xi}\frac{1}{|\lambda|}.
  \]
\end{lemma}

\begin{proof}[Proof of Proposition~\ref{prop:vdc}]
To simplify notation, for $|n|\neq |m|$,
  we write $I_{m,n}=I_{m,n}(T)$ that we rewrite as
\[
   I_{m,n}=\abs{\int_0^T e^{\ui \lambda_{m,n}\,\phi_{m,n}(t)}\d t}
\]
where $\lambda_{m,n}=2 \pi(|n|^s-|m|^s)$ and 
\[
   \phi_{m,n}(t)=t+\frac{n-m}{|n|^s-|m|^s}\, p(t).
\]

\proofskip{\bf Step 1:} For
$(n,m)\in S_{\mathrm{good}}$,
set
\[
\phi_{m,n}'(t) = 1 + \frac{n-m}{|n|^s - |m|^s}\, p'(t).
\]
In this case, $\phi_{m,n}'$ has a positive lower bounded that depends on the sub-regions $S_{\mathrm{good}}^\pm$.
More precisely:
\begin{itemize}
\item[--] If $(n,m)\in S^{+}_{\mathrm{good}}$, then 
\(\dfrac{|n|^s - |m|^s}{n-m} > 0\). Since \(p'(t)\ge 0\), we infer that 
$\phi_{m,n}'(t) \ge 1$ for all  $t \in [0,T]$.
\item[--] If $(n,m)\in S^{-}_{\mathrm{good}}$, then 
\(\dfrac{|n|^s - |m|^s}{n-m} \le -\tau\). Let
\begin{equation}
    A := \Big|\frac{n-m}{|n|^s - |m|^s}\Big| \le \frac{1}{\tau}.\label{def-A}
\end{equation}
Since \(\sup_{t\in[0,T]} p'(t) \le c_2 T^{\alpha-1}\) and by
definition \(\tau = 2 c_2 T^{\alpha-1}\), we get
\[
\phi_{m,n}'(t) = 1 - A p'(t) \ge 1 - \frac{\sup p'(t)}{\tau} \ge \frac{1}{2}.
\]
\end{itemize}
In both cases, $\phi_{m,n}'(t)$ is monotone and bounded away from zero.  
Therefore, Van der Corput's Lemma~\ref{lem:vdc} implies
that $ I_{m,n}(T) \lesssim \bigl||n|^s - |m|^s\bigr|^{-1}$.

\proofskip{\bf Step 2:} Next, we discuss $I_{-n,n}$ for
$n\not=0$. Consider $\phi(t)=4\pi p(t)$ so that
\[
I_{-n,n}=\abs{\int_0^T e^{\ui  n\phi(t)}\d t}
   \leq \delta+\abs{\int_\delta^T e^{\ui  n\phi(t)}\d t}
\]
where $\delta<T$ is a parameter to be fixed later. As
$\phi'(t)\geq 4\pi c_1 t^{\alpha-1}\geq 4\pi c_1 \delta^{\alpha-1}$ on
$[\delta,T]$, Van der Corput's Lemma~\ref{lem:vdc} gives
\[
I_{-n,n}
   \leq \delta+\frac{3}{4\pi c_1 \delta^{\alpha-1}|n|}.
\]
Optimizing over $\delta>0$, {\it i.e.} choosing
$\delta=\left(\dfrac{3(\alpha-1)}{4\pi
    c_1|n|}\right)^{\frac{1}{\alpha}}$ we obtain
$I_{-n,n}\lesssim |n|^{-\frac{1}{\alpha}}$.
%
%
However, this choice requires either 
\[ 
T\geq \left(\dfrac{3(\alpha-1)}{4\pi c_1}\right)^{\frac{1}{\alpha}}
\] 
or \[|n|\geq N_T=\dfrac{3(\alpha-1)}{4\pi c_1 T^\alpha}.\]

\proofskip {\bf Step 3:} We conclude with the more delicate case
$(n,m)\in S_{\mathrm{bad}}$. As before we write
\[
\phi_{m,n}'(t) = 1 + \frac{n-m}{|n|^s - |m|^s}\, p'(t) := 1 - A\, p'(t), 
\]
Since $-\tau < \frac{|n|^s - |m|^s}{n-m} < 0$, $\phi'$ might vanish
for some $t\in[0,T]$. Observe also that Assumption~\ref{ass:H-alpha} implies
that $p'$ is monotonic so that $\phi'$ vanishes at most once.

To handle this case, we introduce a parameter
$0<\beta\leq\dfrac{1}{2}$ to be fixed later and define
\[
   J_\beta := \{ t \in (0,\infty) : |\phi_{m,n}'(t)| < \beta \}.
\]
By the monotonicity of $p'$ and the Mean Value Theorem, $J_\beta$ is
an interval $(t_{\min}, t_{\max})$ with $\phi_{m,n}'(t_{\min})=\beta$.
As a first consequence, we have
\[
  A^{-1} (1-\beta) =  p'(t_{\min}) \le   c_2 \,t_{\min}^{\alpha-1}
\]
from which we retain $t_{\min}\geq\left(\dfrac{1}{2 A c_2}\right)^{\frac{1}{\alpha-1}} >0$ on
  the way.  Next, we observe that $\phi_{m,n}'(t_{\max})=-\beta$, giving
\[
  A^{-1} (1+\beta) =  p'(t_{\max}) \ge  c_1 \,t_{\max}^{\alpha-1}.
\]
We infer $t_{\max} \leq \left(\dfrac{3}{2 A c_1}\right)^{\frac{1}{\alpha-1}}$.  The
  Mean Value Theorem gives us therefore a $\xi\in(t_{\min},t_{\max})$
  such that
\begin{align*}
  2\beta = \phi_{m,n}'(t_{\min}) - \phi_{m,n}'(t_{\max})
  = & \; A \bigl( p'(t_{\max})-p'(t_{\min}) \bigr)
  = A \,  p''(\xi)(t_{\max}-t_{\min})\\
  \geq& \; c_3 A \, \xi^{\alpha-2} \; |J_\beta|\\
  \ge & \; c_3 A \, \min( t_{\min}^{\alpha-2}, t_{\max}^{\alpha-2} ) \; |J_\beta|\\
  \gtrsim & \; A^{\frac{1}{\alpha-1}} \, |J_\beta|,
\end{align*}
where we used the previous estimates for $t_{\min}$ and $t_{\max}$.
From this, we conclude that
\begin{equation}
\abs{\int_{J_\beta\cap [0,T]}e^{\ui \lambda_{m,n}\,\phi_{m,n}(t)}\d t}
\leq|J_\beta|
\lesssim A^{-\frac{1}{\alpha-1}}\beta.
\label{eq:vdca2}
\end{equation}
Note that $J_\beta$ may not be included in
$[0,T]$ and that $[0,T]\setminus
J_\beta$ consists of one or two intervals on which
$\phi'$ is monotone. Therefore Van der Corput's Lemma~\ref{lem:vdc}
implies
\begin{equation}
\abs{\int_{[0,T]\setminus J_\beta}e^{\ui \lambda_{m,n}\,\phi_{m,n}(t)}\d t}
\leq \frac{6}{|\lambda_{m,n}|\beta}\lesssim \frac{1}{\abs{|n|^s-|m|^s}\beta}.
\label{eq:vdca1}
\end{equation}
Grouping \eqref{eq:vdca2} and \eqref{eq:vdca1}, we obtain
\begin{equation}
\label{eq:vdcatmp}
\abs{\int_{[0,T]}e^{\ui \lambda_{m,n}\,\phi_{m,n}(t)}\d t}
\lesssim A^{-\frac{1}{\alpha-1}}\beta+\frac{1}{\abs{|n|^s-|m|^s}\beta}.
\end{equation}
Next, we use that $A^{-\frac{1}{\alpha-1}} < \tau^{\frac{1}{\alpha-1}} \lesssim T$.
Thus, 
\[
A^{-\frac{1}{\alpha-1}}=\ent{A^{-\frac{1}{\alpha-1}}}^{\eta}\ent{A^{-\frac{1}{\alpha-1}}}^{1-\eta}
\lesssim A^{-\frac{\eta}{\alpha-1}}T^{1-\eta}
\]
provided $1-\eta\geq 0$. In other words, we can trade some powers of $A$ to some powers of $T$.
Replacing $A$ by its value, this reads
\begin{equation}
\label{eq:optimisation}
\abs{\int_{[0,T]}e^{\ui \lambda_{m,n}\,\phi_{m,n}(t)}\d t}
\lesssim\abs{\frac{|n|^s-|m|^s}{n-m}}^{\frac{\eta}{\alpha-1}}T^{1-\eta}\beta+\frac{1}{\abs{|n|^s-|m|^s}\beta}.
\end{equation}
The expression in \eqref{eq:optimisation} is of the form $h(\beta) = aT^{1-\eta}\beta + \frac{b}{\beta}$
which is optimal for $\beta^* = a^{-\frac{1}{2}}T^{-\frac{1-\eta}{2}}b^{\onehalf }$.
As we require $\beta\leq\dfrac{1}{2}$, the optimal choice may not be appropriate
when $1-\eta=0$ or require $T$ to be large when $1-\eta>0$.
We will therefore take a $\beta$ of the form $\beta_*=C\beta^*=Ca^{-\frac{1}{2}}T^{-\frac{1-\eta}{2}}b^{\onehalf }$ 
for which $h(\beta_*)=(C+C^{-1})\bigl(abT^{1-\eta})^{\onehalf }$.
We now distinguish several cases.

\proofskip{\bf Case 1.} $\eta=1$.\hfill\\
In this case, the right hand side of \eqref{eq:optimisation} does not
depend on $T$, so that we may optimize on $\beta$ alone.  We choose
\[
  \beta =\beta_{n,m} :=
   C\left(\frac{\abs{n-m}}{\abs{|n|^s-|m|^s}^{\alpha}}\right)^{\frac{1}{2(\alpha-1)}},
\]
where $C$ is a constant, independent of $m,n$, that is chosen to ensure
$\beta\leq\dfrac{1}{2}$. Such a constant exists if and only if
$\dfrac{\abs{n-m}}{\abs{|n|^s-|m|^s}^{\alpha}}$ is bounded
above. According to Lemma~\ref{lem:bounded}, this is the case exactly
when $(s-1)\alpha\geq 1$. With this choice of $\beta$, we obtain
\[
\abs{\int_{[0,T]}e^{\ui \lambda_{m,n}\,\phi_{m,n}(t)}\d t}
\lesssim  \;   \left( \frac{ \bigl|  |n|^s-|m|^s \bigr|^{2-\alpha} }{|n-m|}  \right)^{\frac{1}{2(\alpha-1)}},
\]
which is the claimed estimate.

\proofskip{\bf Case 2.} $\eta<1$.\hfill\\
In this case the choice of $T>0$ in \eqref{eq:optimisation} plays a role.
Define
\begin{align}
  \beta_{n,m}(T)
  := & \; \frac{1}{2} \abs{\frac{|n|^s-|m|^s}{n-m}}^{-\frac{\eta}{2(\alpha-1)}}\Bigl| |n|^s-|m|^s \Bigr|^{-\frac{1}{2}} T^{-\frac{1-\eta}{2}}\notag\\
  =  & \; \frac{1}{2}\left(\frac{\abs{n-m}^{\eta}}{\bigl| |n|^s-|m|^s \bigr|^{\alpha-1+\eta}}\right)^{\frac{1}{2(\alpha-1)}}T^{-\frac{1-\eta}{2}}.\label{eq:beta}
\end{align}
Assume now that
\begin{equation}
\label{eq:goodconditionbad1}
\mm(s,\alpha,\eta):=\sup_{(m,n)\in S_{\mathrm{bad}}}\left(\frac{\abs{n-m}^{\eta}}{\abs{|n|^s-|m|^s}^{\alpha-1+\eta}}\right)^{\frac{1}{2(\alpha-1)}}<+\infty
\end{equation}
then, for 
\begin{equation}
\label{eq:goodconditionbad2}
T\geq T_1\geq\mm(s,\alpha,\eta)^{\frac{2}{1-\eta}},
\end{equation}
we have $\beta_{n,m}(T)\leq \dfrac{1}{2}$. Then choosing $\beta=\beta_{n,m}(T)$, \eqref{eq:optimisation}
reads
\begin{align*}
\abs{\int_{[0,T]}e^{\ui \lambda_{m,n}\,\phi_{m,n}(t)}\d t}
\lesssim& \; \abs{\frac{|n|^s-|m|^s}{n-m}}^{\frac{\eta}{2(\alpha-1)}}\frac{1}{\abs{|n|^s-|m|^s}^{\frac{1}{2}}} T^{\frac{1-\eta}{2}}\\
\lesssim& \; \frac{T^{\frac{1-\eta}{2}}}{|n-m|^{\frac{\eta}{2(\alpha-1)}}\abs{|n|^s-|m|^s}^{\frac{\alpha-1-\eta}{2(\alpha-1)}}},
\end{align*}
which is the claimed inequality. 
It remains to find conditions on $s,\alpha,\eta$ that imply that $\mm(s,\alpha,\eta)<+\infty$.
We distinguish three sub-cases for $\eta$.

\proofskip{\bf Case 2a.} {$0<\eta<1$.}\hfill\\
Set $\gamma=1+\dfrac{\alpha-1}{\eta}$ so that
\[
\left(\frac{\abs{n-m}^{\eta}}{\abs{|n|^s-|m|^s}^{\alpha-1+\eta}}\right)^{\frac{1}{2(\alpha-1)}}
=\left(\frac{\abs{n-m}}{\abs{|n|^s-|m|^s}^{\gamma}}\right)^{\frac{\eta}{2(\alpha-1)}}.
\]
Thus $\mm(s,\alpha,\eta)\leq M(\gamma,s)^{\frac{\eta}{2(\alpha-1)}}$
with $M(\gamma,s)$ from Lemma~\ref{lem:bounded}. According to that lemma, $M(\gamma,s)<+\infty$
precisely when
\[
  (s-1)\gamma\geq 1\quad\Longleftrightarrow\quad(s-1)(\alpha-1)\ge
  (2-s)\eta.
\]
When $s\geq2$, this is satisfied for every $\eta > 0$. 

\noindent For values $1<s<2$, this leads to the restriction
$\eta\leq \dfrac{s-1}{2-s}(\alpha-1)$. But $\dfrac{s-1}{2-s}(\alpha-1)<1$
only for $1<s<1 + \frac{1}{\alpha}$ so that for $s\ge 1 + \frac{1}{\alpha}$ all $\eta\in(0,1)$ are allowed.

In summary, if one of the following conditions holds
\[\left\{
\begin{array}{lcl}
  s\geq 1 + \frac1{\alpha}  & \text{and} &\ 0<\eta<1\\
  s < 1 + \frac1{\alpha} & \text{and} &  0 <\eta \le  \dfrac{s-1}{2-s}(\alpha-1)
\end{array}\right.
\]
then \eqref{eq:goodconditionbad1} holds and we can chose $T_1= 4^{\frac{\eta}{(1-\eta)(\alpha-1)}}\geq
M(\gamma,s)^{\frac{\eta}{(1-\eta)(\alpha-1)}}\geq \mm(s,\alpha,\eta)^{\frac{2}{1-\eta}}$. Note that $T_1$ can be chosen arbitrarily near to $1$
by taking $\eta>0$ small enough.

\proofskip{\bf Case 2b.} {$\eta=0$.}\hfill\\
In this case, as $|n|\neq|m|$,
\[
  \beta_{m,n}(T)=\frac{1}{2}\frac{T^{-\frac{1}{2}}}{\abs{|n|^s-|m|^s}^{\frac{1}{2}}}
  \leq\dfrac{1}{2}
\]
when $T\geq T_1=1$.

\proofskip{\bf Case 2c.} {$\eta<0$.}\hfill\\
 In this case, it is better to introduce
 $\gamma=1-\dfrac{\alpha-1}{|\eta|}$ and rewrite \eqref{eq:goodconditionbad1}
 as
  \begin{equation}
     \label{eq:condT1b}
\mm(s,\alpha,\eta):=\sup_{(m,n)\in S_{\mathrm{bad}}}
\left(\frac{\abs{|n|^s-|m|^s}^\gamma}{\abs{n-m}}\right)^{\frac{|\eta|}{2(\alpha-1)}}
 \end{equation}
 If $1-\alpha\leq\eta<0$, then $\gamma\leq0$.
Since $\abs{|n|^s-|m|^s}\geq1$ and $\abs{n-m}\geq1$,
\eqref{eq:condT1b} gives $\mm(s,\alpha,\eta)\leq1$,
uniformly in $T$. Thus \eqref{eq:goodconditionbad2} holds
as soon as $T\geq T_1=1$.
 Finally, Table~\ref{table:eta} is obtained by summarizing all cases.
\end{proof}

\subsection{A third technical lemma}\label{Sec:ser}

\begin{lemma}\label{lem:sum-unified}
  Let $s>1$, $\delta>0$ and $\gamma\in\mathbb R$. Assume the
  parameters satisfy the condition
\[
s\delta+\gamma \;>\;
\begin{cases}
1, &\text{if }\gamma\ge0,\\[2pt]
\max(1,\delta), &\text{if }\gamma<0.
\end{cases}
\]
Further, for $\gamma\ge 0$ define
\[
\sigma=\begin{cases}(s-1)\delta&\mbox{if }\gamma+\delta>1,\\
s\delta+\gamma-1&\mbox{if }\gamma+\delta<1,\\
(s-1)\delta-\eps&\mbox{if }\gamma+\delta=1\mbox{ with }0<\eps<(s-1)\delta,\end{cases}
\]
and for $\gamma<0$ define
\[
\sigma=\begin{cases}
    s\delta+\gamma-1 &\mbox{if }\delta<1,\\
    (s-1)\delta+\gamma &\mbox{if }\delta>1,\\
    s+\gamma-1-\varepsilon &\mbox{if } \delta=1 \mbox{ with } 0<\varepsilon<s+\gamma-1.
\end{cases}
\]
For integers \(m,n\in\Z\) with $|m|\not=|n|$, define
\begin{equation}  \label{eq:a-n-m}
a_{n,m}:=\frac{1}{|n-m|^\gamma\; \bigl||n|^s-|m|^s\bigr|^{\delta}},
\end{equation}
Then the following holds.

\begin{aufzi}
\item\label{item:tail-decay} (Uniform tail decay) For every integer \(N\ge1\) define the tail
  \[
      S_m(N):=\sum_{\substack{n\in\mathbb Z\\ |n|\ge N,\ |n|\ne|m|}} a_{n,m}.
  \]
  Then there exist constants \(C>0\) and $N_0$ (depending only on
  \(\gamma,\delta,s\) and on  the chosen \(\varepsilon\) in the logarithmic cases) such that for all \(N\geq N_0\) and all
  integers \(m\),
  \[
    S_m(N)\le C\,N^{-\sigma}.
  \]
In particular, \(S_m(N)\to0\) as
\(N\to\infty\) uniformly in \(m\). In particular, there is a constant
  \(C(\gamma,\delta,s)\) such that \(|S_m|\le C\) for all \(m\).
\item\label{item:global-bdd} (Decay of $S_m$) For every
  \(m\in\mathbb Z\) the series
  \(\dst S_m:=\sum_{\substack{n\in\mathbb Z\\|n|\ne|m|}} a_{n,m}\) converges and
  \(\lim_{|m|\to\infty} S_m =0\). 
\end{aufzi}
\end{lemma}
\begin{proof}
\textbf{~\ref{item:tail-decay} Tail decay.}  First
observe that the condition $\gamma+s\delta>1$ implies that $S_m$
converges for each $m$. Now fix \(m\in\mathbb Z\). Observe first that  
$a_{-n,-m}=a_{n,m}$ so that $S_{-m}=S_m$. We can thus assume that $m\geq 0$.
Next, for $0\leq m\leq 10$, $N\geq 20$ and $|n|\geq N$,
\[
a_{n,m}=\frac{1}{|n-m|^\gamma\; \bigl||n|^s-|m|^s\bigr|^{\delta}}\lesssim
\frac{1}{|n|^{\gamma+\delta s}}
\]
so that
  \[
    S_m(N) \lesssim \sum_{|n| > N} \frac{1}{|n|^{\gamma+\delta s}}  \lesssim N^{1-(\gamma+s\delta)}\lesssim N^{-\sigma}
  \]
for any $\sigma\leq \gamma+s\delta-1$.  We are left with $m\geq 10$. 
Now, for $N\geq 20$, write 
\begin{align*}
S_m(N) =&
\; \sum_{\substack{n\ge N\\ n\ne m}} \frac{1}{|n-m|^\gamma\,|n^s-m^s|^\delta}
+ \sum_{\substack{n\le -N\\ n\ne -m}} \frac{1}{|n-m|^\gamma\,||n|^s-m^s|^\delta}\\
= &\; \sum_{\substack{n\ge N\\ n\ne m}} \frac{1}{|n-m|^\gamma\,|n^s-m^s|^\delta}  
+ \sum_{\substack{n\ge N\\ n\ne m}} \frac{1}{(n+m)^\gamma\,|n^s-m^s|^\delta}\\
\leq & \; 2\sum_{\substack{n\ge N\\ n\ne m}} \frac{1}{|n+\eta m|^\gamma\,|n^s-m^s|^\delta}
\end{align*}
where $\eta=+1$ if $\gamma< 0$ and $\eta=-1$ when $\gamma\geq 0$.
Call this last sum $T_m(N)$ and split it into the {\it ``near part''} \(n\in[N,2m]\) and the {\it ``far part''}
\(n\ge\max(N,2m+1)\), 
\(T_m(N)=T^{\mathrm{near}}(N)+T^{\mathrm{far}}(N)\).

\smallskip

\medskip\noindent\textbf{Far part.} If \(n\ge 2m+1\) then \(n/2\ge m\)
and therefore, with \(c_s:=1-2^{-s}>0\),
\[
n^s-m^s \;=\; n^s - m^s \;\ge\; n^s - (n/2)^s = c_s n^s,
\]
while
\[
\frac{n}{2}\leq  |n+\eta m|\le \frac{3n}{2}.
\]
Hence there exists an absolute constant \(C_1=C_1(s,\delta,\gamma)\)
such that
\[
\frac{1}{|n+\eta m|^\gamma\,|n^s-m^s|^\delta} \le C_1\, \frac{1}{n^{\gamma+s\delta}}.
\]
Therefore the far tail satisfies
\[
T^{\mathrm{far}}(N) \le C_1 \sum_{n\ge \max(N,2m+1)} \frac{1}{n^{\gamma+s\delta}}
 \le C_1 \sum_{n\ge N} \frac{1}{n^{\gamma+s\delta}} \lesssim N^{1-(\gamma+s\delta)}.
\]
We conclude that $T^{\mathrm{far}}(N)\lesssim N^{-\sigma}$ provided $\sigma\leq s\delta+\gamma-1$.

\smallskip

\medskip\noindent\textbf{Near part.} If $N>2m$, the near part is empty. Otherwise, $m\ge N/2$.
For the near indices $N\leq n\leq2m$, write $j=|n-m|$, so that
$n=m+j$ or $n=m-j$ and $1\leq j\leq m$. Put $r=n/m$. Then
\[
\frac{|n^s-m^s|}{m^{s-1}|n-m|}
=\frac{|r^s-1|}{|r-1|}.
\]
The quotient extends continuously at $r=1$ with value $s$ and is
positive on $[0,2]$. Thus
\[
|n^s-m^s|^\delta\asymp m^{(s-1)\delta}j^\delta,
\]
with constants depending only on $s,\delta$.

\proofskip Let us first consider the case $\gamma< 0$. In this case, as $n\leq 2m$,
\[
\frac{1}{|n+m|^\gamma\,|n^s-m^s|^\delta} \lesssim \, m^{-(s-1)\delta-\gamma}\,
\frac{1}{j^{\delta}}.
\]
Summing over \(j\) from \(1\) up to \(Cm\) gives
\[
T^{\mathrm{near}} \lesssim\,m^{-(s-1)\delta-\gamma}\sum_{j=1}^{Cm}\frac{1}{j^{\delta}}
\lesssim \left\{\begin{array}{ll}m^{-(s\delta+\gamma-1)}&\mbox{if }\delta <1\\
m^{-(s\delta+\gamma-\delta)}&\mbox{if }\delta >1\\
m^{1-(s+\gamma)}\log m&\mbox{if }\delta =1
\end{array}\right..
\]
Now choose 
\[
\sigma=\begin{cases}s\delta+\gamma-1&\mbox{when }\delta<1\\
s\delta+\gamma-\delta&\mbox{when }\delta>1\\
s+\gamma-1-\eps&\mbox{when }\delta=1\mbox{ with }0<\eps<s+\gamma-1\end{cases}
\]
and notice that $0<\sigma\leq  s\delta+\gamma-1$.
Then, using $m\geq N/2$ we obtain $T^{\mathrm{near}} \lesssim N^{-\sigma}$
while $T^{\mathrm{far}} \lesssim N^{-\sigma}$
still holds.

\proofskip In the case that $\gamma\geq0$  there is a
constant \(C_2\) with
\[
\frac{1}{|n-m|^\gamma\,|n^s-m^s|^\delta} \le C_2\, m^{-(s-1)\delta}\,
\frac{1}{j^{\gamma+\delta}}.
\]
Summing over \(j\) from \(1\) up to \(Cm\) gives
\[
T^{\mathrm{near}} \lesssim
m^{-(s-1)\delta}\sum_{j=1}^{Cm}\frac{1}{j^{\gamma+\delta}}
\lesssim \left\{\begin{array}{ll}\dfrac{1}{m^{s\delta+\gamma-1}}&\mbox{if }\gamma+\delta<1\\
\dfrac{1}{m^{(s-1)\delta}}&\mbox{if }\gamma+\delta>1\\[6pt]
\dfrac{\log m}{m^{(s-1)\delta}}&\mbox{if }\gamma+\delta=1
\end{array}\right..
\]
Choosing the same $\sigma$ as in the case $\gamma\ge 0$, we again obtain that $T^{\mathrm{near}} \lesssim N^{-\sigma}$.

This completes the proof of the first part of the lemma.

\smallskip

\medskip\proofskip \textbf{\ref{item:global-bdd} Decay of $S_m$.}  
Uniform boundedness of $S_m$ follows from the tail estimate. But 
we can actually prove a stronger decay estimate of $S_m$.  By symmetry
$S_{-m}=S_m$, so it suffices to consider $m\ge 42$ (say). Now we write
\[
S_m=\sum_{\substack{n\in\Z\\|n|\ne m}}
\frac{1}{|n-m|^\gamma\;|\,|n|^s-m^s|^\delta}
\leq \frac{1}{m^{\gamma+s\delta}} + 2\sum_{\substack{n\ge1\\ n\ne m}}
\frac{1}{|n+\eta m|^\gamma\;|n^s-m^s|^\delta}
\]
where again $\eta=+1$ is $\gamma<0$ and $\eta=-1$ if $\gamma\geq 0$.
It is enough to bound the last sum
\[
\Sigma_m := \sum_{\substack{n\ge1\\ n\ne m}} \frac{1}{|n+\eta m|^\gamma\;|n^s-m^s|^\delta}.
\]
To do so, we cut it into four parts
\begin{equation}\label{eq:b-3}
  \Sigma_m  =  \; \sum_{n=1}^{\lfloor \frac{m}{2}\rfloor}+\sum_{n=\lfloor \frac{m}{2}\rfloor+1}^{m-1}+\sum_{n=m+1}^{2m}+\sum_{n=2m+1}^{+\infty}\\
            = \; \mathrm{I} + \mathrm{II} + \mathrm{III} + \mathrm{IV} 
\end{equation}
and estimate them separately.

\proofskip \textbf{Range I: \(1\le n\le \lfloor m/2\rfloor\).} Observe
that $|n\pm m|\asymp m$ and $m^s-n^s \ge m^s
(1-2^{-s})$. Consequently,
\[
  \mathrm{I}\lesssim \frac{1}{m^\gamma}      \sum_{n=1}^{\lfloor \frac{m}{2}\rfloor} \frac{1}{(m^s-n^s)^\delta}
           \asymp \frac{1}{m^{s\delta+\gamma-1}}.
\]
This term tends to zero, since $s\delta+\gamma-1>0$.

\proofskip  \textbf{Range II: \(\lfloor m/2\rfloor+1 \le n \le m-1\).} 
We write $j=m-n$ and observe  $1\leq j\leq\dfrac{m}{2}$.
The Mean Value Theorem then shows that, for each such $j$, there is a $\xi_j$ satisfying $j\leq m-j<\xi_j<m$ such that
\begin{equation*}
    (m^s-n^s)^\delta=\bigl(m^s-(m-j)^s\bigr)^\delta =s^\delta \xi_j^{(s-1)\delta}j^\delta
    \asymp m^{(s-1)\delta}j^\delta.
\end{equation*}

\proofskip{\bf Case 1.} $\gamma \ge 0$.\\
In this case
\[
      \mathrm{II}= \sum_{n=\lfloor \frac{m}{2}\rfloor+1}^{m-1}\frac{1}{(m-n)^\gamma(m^s-n^s)^\delta}
\lesssim \frac{1}{m^{(s-1)\delta}}\sum_{1\leq j\leq\frac{m}{2}}\frac{1}{j^{\delta+\gamma}}.
\]
Estimating the sum gives
\[
\mathrm{II}\lesssim\left\{\begin{array}{ll}
           m^{1-\gamma-\delta s}  & \text{if}\ \delta+\gamma<1\\
         \ln(m) m^{1-\gamma-\delta s}  & \text{if}\ \delta+\gamma=1\\
         m^{\delta(1-s)} & \text{if}\ \delta+\gamma>1
     \end{array}\right.
\]
which tends to zero with $m\to\infty$  in either case, using $s \delta + \gamma > 1$, $\delta> 0$  and $s> 1$.

\proofskip {\bf Case 2.} $\gamma < 0$. \hfill\\
 We now have to bound
 \[
\mathrm{II}=       \sum_{n=\lfloor \frac{m}{2}\rfloor+1}^{m-1}\frac{1}{(m+n)^\gamma(m^s-n^s)^\delta}
\lesssim  
       \frac{1}{m^{(s-1)\delta+\gamma}}  \sum_{1\leq j\leq \frac{m}{2}}\frac{1}{j^{\delta}}.
 \]
 Using the standard estimates of the last sum, we obtain
 \[
\mathrm{II} 
 \lesssim 
     \begin{cases}
         \frac{1}{m^{s\delta+\gamma-1}}& \text{if}\; \delta<1,\\
         \frac{\ln(m)}{m^{s+\gamma-1}} & \text{if}\; \delta=1\\
         \frac{1}{m^{s\delta+\gamma-\delta}}  & \text{if}\; \delta>1
     \end{cases}.      
\]
As $\delta s + \gamma > \max(1,\delta)$, in all
cases $\mathrm{II} \to 0$ when $m\to\infty$.

 \proofskip \textbf{Range III: \(m+1\le n \le 2m\).} The case
 $m+1\leq n\leq 2m$ and $j=n-m$ is somehow symmetric to the second
 one. We use again the Mean Value Theorem to infer that for some
 $\xi_j$ with $j\leq m\leq \xi_j\leq m+j\leq 2m$,
\[
    (n^s-m^s)^\delta=\bigl((m+j)^s-m^s\bigr)^\delta
       =s^\delta j^\delta \xi_j^{(s-1)\delta}\asymp m^{(s-1)\delta}j^\delta.
\]
Similarly to the second case, we infer for non-negative $\gamma$
\begin{align*}
        \sum_{n=m+1}^{2m} \frac{1}{(n-m)^\gamma(n^s-m^s)^\delta}\asymp &\,\frac{1}{m^{(s-1)\delta}}\sum_{j=1}^{m} \frac{1}{j^{\gamma+\delta}}
\end{align*}
while for negative $\gamma$
\begin{align*}
        \sum_{n=m+1}^{2m} \frac{1}{(n+m)^\gamma(n^s-m^s)^\delta}\asymp &\,\frac{1}{m^{(s-1)\delta+\gamma }}\sum_{j=1}^{m} \frac{1}{j^{\delta}}.
\end{align*}
In both cases, the asymptotics is similar to the one for $\mathrm{II}$ and $\mathrm{III}$ goes to $0$
as $m\to+\infty$.

\proofskip \textbf{Range IV (tail): \(n\ge 2m+1\).} We write $j=n-m\geq m$. This yields
\[
(n^s- m^s)^\delta=
       \bigl((m+j)^s-m^s\bigr)^\delta
       = j^{s\delta}\left(\left(\frac{m}{j}+1\right)^s- \left(\frac{m}{j}\right)^s\right)^\delta
    \geq j^{s\delta}
\]
since $(1+a)^s-a^s\geq 1$ when $s\geq1$ and $a\geq0$.
As $|m+\eta n|\asymp j$, we obtain
\begin{equation*}
        \sum_{n=2m+1}^\infty \frac{1}{|m+\eta n|^\gamma (n^s-m^s)^\delta}
        \lesssim  \sum_{j=m+1}^\infty \frac{1}{j^{s\delta+\gamma}}
\end{equation*}
which vanishes at infinity, since we assumed $s\delta+\gamma>1$.

In conclusion, all four terms in \eqref{eq:b-3} go to $0$ when $|m|\to+\infty$
and the proof is done.
\end{proof}

\subsection{Proof of the curved Ingham inequality}
\label{Sec:pf:schroedinger}

We can now prove Theorem~\ref{thm:schroedinger}.

\begin{proof}
Consider the series
\begin{equation*}
        u(t,x)=\sum_{n\in\Z}c_n e^{2\ui\pi  (nx+|n|^st)}.
    \end{equation*}
It is enough to prove the theorem when $c_n=0$ when $|n|$ is (arbitrarily) large. A standard density argument will then
allow to conclude. We then write
\begin{equation}\label{eq:keythm-2}
\begin{aligned}
       \int_0^T \abs{u\bigl(t,p(t)\bigr)}^2\,\mathrm{d}t
= &\; \sum_{n.m\in\Z} c_n\overline{c_m} \int_0^T  
       e^{2\ui\pi\bigl( (n-m)p(t) + (|n|^s-|m|^s) t \bigr)} \, \mathrm{d}t\\
 =&\; T\sum_{n\in\Z}  |c_n|^2  +  \sum_{\substack{n,m\in\Z\\ n\not=m}}  c_n\overline{c_m} \, J_{n,m}
\end{aligned}
\end{equation}    \ignorespacesafterend
where $(J_{n,m})$ are defined as in \eqref{eq:without-abs} and $|J_{n,m}|=I_{n,m}$ for the quantities in Proposition~\ref{prop:vdc}.
We further assume that $T\geq \max(T_0, T_1)$ so that we
may estimate $I_{n,m}$ with this proposition. We further cut
this sum into 3 pieces, corresponding to the different
estimates for the cases  $n=-m$, $(n,m)\in S_{\mathrm{good}}$,
and $(n,m)\in S_{\mathrm{bad}}$.

\proofskip First
\begin{equation}\label{eq:sum-1}
\biggl|\sum_{n\in\Z\setminus\{0\}}c_n\overline{c_{-n}}\,J_{n,-n}\biggr|
\lesssim\sum_{n\in\Z\setminus\{0\}}\frac{|c_n||c_{-n}|}{|n|^{1/\alpha}}
\lesssim\sum_{n\in\Z}|c_n|^2
\end{equation}
with Cauchy-Schwarz and $|n|^{1/\alpha}\geq 1$.

\proofskip Next we obtain
\begin{align}\label{eq:sum-2}
\biggl|\sum_{(n,m)\in S_{\mathrm{good}}}c_n\overline{c_{m}} \, J_{n,m}\biggr|
\lesssim&\,\sum_{(n,m)\in S_{\mathrm{good}}}\frac{|c_n||c_{m}|}{\bigl||n|^s-|m|^s\bigr|}
\leq\sum_{(n,m)\in S_{\mathrm{good}}}\frac{|c_n|^2+|c_{m}|^2}{2\bigl||n|^s-|m|^s\bigr|}\notag\\
\leq&\,\sum_{n\in\Z}|c_n|^2\sum_{m:m\neq \pm n}\frac{1}{\bigl||n|^s-|m|^s\bigr|}\notag\\
\leq&\,\left(\sup_{n\in\Z}\sum_{m:m\neq \pm n}\frac{1}{\bigl||n|^s-|m|^s\bigr|}\right)\sum_{n\in\Z}|c_n|^2\notag\\
\lesssim&\, \sum_{n\in\Z}|c_n|^2,
\end{align}
using Lemma~\ref{lem:sum-unified}~\ref{item:global-bdd} with $\gamma=0,\delta=1$.

\proofskip Now we consider the case of $(n,m)\in S_{\mathrm{bad}}$. By
Proposition~\ref{prop:vdc} (iii) we have
\begin{equation*}
    I_{n,m}\lesssim T^{\frac{1-\eta}{2}} |n-m|^{-\frac{\eta}{2(\alpha-1)}}\bigl| |n|^s-|m|^s  \bigr|^{-\frac{\alpha-1-\eta}{2(\alpha-1)}}
\end{equation*}
for any $\eta$ satisfying the conditions of Table~\ref{table:eta}. Define
\begin{equation}\label{eq:s-delta-gamma}
  \gamma=\frac{\eta}{2(\alpha-1)} \quad \text{and} \quad \delta= \frac{\alpha-1-\eta}{2(\alpha-1)},
  \quad\text{so that}\quad
  s\delta + \gamma = \dfrac{s(\alpha-1)-(s-1)\eta }{2(\alpha-1)}.
\end{equation}
For $\eta$ satisfying Table~\ref{table:eta}, we have the
estimate
\begin{align*}
    \biggl|\sum_{(n,m)\in S_{\mathrm{bad}}}c_n\overline{c_m} J_{n,m}\biggr|\lesssim &\, T^{\frac{1-\eta}{2}}\sum_{(n,m)\in S_{\mathrm{bad}}}\frac{|c_n||c_m|}{|n-m|^\gamma\bigl||n|^s-|m|^s \bigr|^\delta}\\
  \le&\, T^{\frac{1-\eta}{2}}\sum_{(n,m)\in S_{\mathrm{bad}}}\frac{|c_n|^2+|c_m|^2}{2|n-m|^\gamma\bigl||n|^s-|m|^s\bigr|^\delta}\\
   = &\, T^{\frac{1-\eta}{2}}\sum_{(n,m)\in S_{\mathrm{bad}}}\frac{|c_n|^2}{|n-m|^\gamma\bigl||n|^s-|m|^s\bigr|^\delta}
\end{align*}
where we exploited the symmetry in $n, m$, so that
\begin{equation}  \label{eq:sum-3}
  \biggl|\sum_{(n,m)\in  S_{\mathrm{bad}}}  c_n\overline{c_m} J_{n,m}\biggr|\lesssim
      K\, T^{\frac{1-\eta}{2}} \; \sum_{n\in\Z}|c_n|^2.
\end{equation}
where \begin{equation} \label{eq:constant}
    K := \sup_{n \in \Z} \left(   \sum_{\substack{(n,m) \in     \Z^2 \\ m\neq \pm n}}   \frac{1}{|n-m|^\gamma\bigl||n|^s-|m|^s
      \bigr|^\delta}  \right).
\end{equation}

Assume for a moment that we have already established $K < \infty$.
Then considering \eqref{eq:sum-1}, \eqref{eq:sum-2} and
\eqref{eq:sum-3} together into \eqref{eq:keythm-2}, we obtain
\begin{aufzi}
\item the upper bound
  \[
    \int_0^T|u(t,p(t))|^2\d t\lesssim \bigl(T+T^{\frac{1-\eta}2}+1\bigr)\sum_{n\in\Z}|c_n|^2
  \]
which establishes \eqref{eq:obs-curve-schroedinger-upper};
\item the lower bound
  \[
    \int_0^T|u(t,p(t))|^2\d t\gtrsim \bigl(T-C
    (T^{\frac{1-\eta}2}+1)\bigr)\sum_{n\in\Z}|c_n|^2
  \]
  for some positive constant $C$ depending only on $s,\gamma,
  \delta$. If we choose $T$ large enough so that
  \( T-C(T^{\frac{1-\eta}2}+1)>0 \), we obtain the desired inequality
\eqref{eq:obs-curve-schroedinger}, provided that $\eta>-1$.
\end{aufzi}

\proofskip It remains to establish the finiteness of
\eqref{eq:constant}.
 We distinguish two cases.
 
 \proofskip {\bf Case  $s>2$:}  The choice $\eta=0$ leads via \eqref{eq:s-delta-gamma}
 to $\gamma = 0$, $\delta = \dfrac12$ and $s\delta+\gamma = \dfrac{s}2 >1$.
 By  Lemma~\ref{lem:sum-unified}, $K$ is finite.
 
 \proofskip {\bf Case $\dfrac32 < s \le 2$:} In this case,
 \[
   {-}\dfrac{2-s}{s-1} (\alpha-1) > -(\alpha-1)
 \]
 so that we can choose some
 ${-}(\alpha-1) < \eta < {-}\dfrac{2-s}{s-1} (\alpha-1)$, which in turn implies
  via \eqref{eq:s-delta-gamma} that $\delta < 1$   and
 $s\delta+\gamma > \max(1, \delta)$.  Since $\gamma+\delta = \dfrac12$,
 Lemma~\ref{lem:sum-unified} shows that $K$ is finite.
For the lower estimate, when $\frac32<s\leq2$, choose more precisely
\[
\max\left(-(\alpha-1),-1\right)<\eta<
-\frac{2-s}{s-1}(\alpha-1).
\]
This interval is nonempty since $s>\frac{3}{2}$ and
$s>2-\frac{1}{\alpha}$. Hence $K<\infty$ and $\eta>-1$, as required
above. For $s>2$, the choice $\eta=0$ already has both properties.
This completes the proof.
\end{proof}

%
Let us now prove that the condition $T\geq T(p)$ cannot be removed in Theorem~\ref{thm:schroedinger}
if one wants a lower bound with a constant not depending on $T$\,:

\begin{lemma}\label{prop:minimal}
Let $\alpha>1$, $s\ge \dfrac32$. Let
$p\in \cc^1([0,\infty[)\cap \cc^2(]0,\infty[)$ satisfying
Assumption~\ref{ass:H-alpha}. For every $\tilde C>0$, there exists $T_{min}=T(p,\tilde C)$
such that, if the inequality
\begin{equation}
\label{eq:mintimeneeded}
\tilde C \sum_{n\in\Z}|c_n|^2\leq 
\frac{1}{T}\int_0^T\abs{\sum_{n\in\Z}c_ne^{2\ui\pi (np(t)+|n|^st)}}^2\,\mathrm{d} t
\end{equation}
holds for every $(c_n)_{n\in\Z}\in\ell^2(\Z)$ then $T\geq T_{min}$.
\end{lemma}



\begin{proof}
Choose $c_0=1$, $c_1=-e^{-2\pi\ui p(0)}$, and $c_n=0$ otherwise.
Then $\sum_n|c_n|^2=2$. Using $|1-e^{\ui\theta}|\le|\theta|$ and
$|p(t)-p(0)|\le(c_2/\alpha)t^\alpha\leq (c_2/\alpha)t$ when $0\leq t\leq 1$, we obtain
$$
\frac{1}{T}\int_0^T
\left|1-e^{2\pi\ui(p(t)-p(0)+t)}\right|^2\,\mathrm{d}t
\le \frac{4}{3}\pi^2\left(1+\frac{c_2}{\alpha}\right)^2 T^2
\longrightarrow 0.
$$
Thus, for \eqref{eq:mintimeneeded} to hold, one needs 
$\tilde C\leq \frac{2}{3}\pi^2\left(1+\frac{c_2}{\alpha}\right)^2 T^2$
which fails for every sufficiently small $T>0$.
\end{proof}

Alternatively, this proof shows that $\tilde C$ has to depend on $T$ and that $\tilde C(T)\lesssim T^2$.

\section{Strong Ingham Inequalities, the high-frequency and high-dispersion cases}\label{Sec:Th2}

\subsection{The high-frequency case}
\label{Sec:pfTh2}

We will first prove a result that will imply that high-frequency dispersive sums
satisfy Ingham's Inequality without a minimal time condition.

\begin{theorem}\label{thm:highfreqmeas}
Let \(\mu\) be a probability measure on \(\R^2\) with Fourier decay
\[
  | \widehat{\mu}(\xi)| \lesssim \dfrac{1}{(1+|\xi|)^{\delta}}
\]
for some \(\delta>0\).
Let $0<\eta<1$ be such that 
\[
  | \widehat\mu(\xi) |  \le \eta\quad\mbox{for } |\xi| \geq 1.
\]
Assume \(s>1\) and \(s\delta > 1 \). Then there is an $N(s, \mu)$ such that,
for every \(N \ge N(s, \mu)\) and every \((c_n)_{|n|\ge N}\in\ell^2\),
\begin{equation} \label{eq:thmhighfreqNbig}
    \dfrac{1}{2}\sum_{|n|\ge  N}|c_n|^2
    \leq    \int_{\R^2} \abs{\sum_{|n|\geq N}c_n  e^{2\ui\pi(|n|^st+nx)}}^2\,\mathrm{d}\mu(t,x)
    \leq    \dfrac{3}{2}\sum_{|n|\ge N}|c_n|^2.
\end{equation}
For fixed $N\geq2$, there exists $\zeta(\mu,N)$ such that, if $s\geq\zeta(\mu,N)$,
\begin{equation}    \label{eq:thmhighfreqsbig}
  \dfrac{1-\eta}{2}\sum_{|n|\ge N}|c_n|^2
  \leq  \int_{\R^2} \abs{\sum_{|n|\geq N}c_n e^{2\ui\pi(|n|^st+nx)}}^2\,\mathrm{d}\mu(t,x)
  \leq \dfrac{3+\eta}{2}\sum_{|n|\ge N}|c_n|^2.
\end{equation}
\end{theorem}

\begin{remark}
  For any fixed $s_0>\max\left(1,1/\delta\right)$, the proof provides an estimate
  $\zeta(\mu,N)\le s_0+C_\mu/\log N$ with constants depending on
  $\mu$.
\end{remark}

\begin{remark}
For a probability measure $\mu$
on $\R^d$, the quantity
\[
\dim_\ff(\mu):=\sup\{t\in[0,d]\,:\ |\widehat{\mu}(\xi)|\lesssim(1+|\xi|)^{-t/2}\}
\]
is called its Fourier dimension.
The Fourier dimension of a subset
$E\subset\R^d$ is then
\[
\dim_\ff(E):=\sup\{t\in[0,d]\,:
\exists\mu\in\pp(E)\mbox{ s.t. }\dim_\ff(\mu)\ge t\}.
\]
where $\pp(E)$ is the set of probability measures supported on $E$.
The determination of the Fourier dimension of sets is a vast subject
and, we refer to P. Mattila's book \cite{Ma15} and the surveys by
R. Lyons \cite{Ly} and F. Ekström \& J. Schmeling \cite{ES} as
starting points.

Let us here mention that the Fourier dimension is known to be at most the Hausdorff dimension and sets for which they coincide are
called Salem sets.

The middle-third Cantor set is known not to have positive Fourier
dimension but Cantor sets obtained by randomly selecting basic
intervals with spatial independence are in general Salem sets 
(see e.g. \cite{Shm17,SS18}).  Salem sets or sets with positive
Fourier dimensions are ubiquitous in random settings. Kahane
\cite{Kah85} showed that the image of a set under the fractional
Brownian motion is almost surely Salem. Kahane’s result was
generalized to Gaussian random fields with stationary increments
\cite{SX06}, Brownian sheet \cite{KWX06}, and fractional Brownian
sheets \cite{WX07}.  There are also some non-random constructions of
sets of positive Fourier dimension based on self-affine measures
\cite{LiSahlsten}, iterated function systems \cite{BakerBanaji} and
number theory \cite{FH}.
\end{remark}

\begin{remark}
It is well known that $\widehat{\mu}$ is continuous and positive definite. If there is a $\xi_0\neq 0$ such that $|\widehat{\mu}(\xi_0)|=1$, then the triangle inequality implies that $e^{-2\pi i\langle z,\xi_0\rangle}$ is constant $\mu$-almost everywhere. Hence $|\widehat{\mu}(k\xi_0)|=1$ for every integer $k$. As we assumed that
  $\widehat{\mu}(\xi)\to 0$ when $\xi\to\infty$, this can not happen. As a consequence,
  there is indeed some $0<\eta<1$ such that $|\widehat{\mu}(\xi)|\leq\eta$ for $|\xi|\geq 1$.
\end{remark}

\begin{proof}[Proof of Theorem~\ref{thm:highfreqmeas}]
  Let $\phi$ be given by $\phi(n)=(|n|^s,n)\in \mathbb{R}^2$ and $N \ge 1$. 
  We use the convention $\widehat{\mu}(\xi)=\int_{\R^2}e^{-2\pi \ui \scal{z,\xi}}\d \mu(z)$.
  Then
\begin{equation}
\begin{aligned}
\int_{\R^2} & \Bigl| \sum_{|n|\geq N}c_n e^{2\ui\pi(|n|^st+nx)} \Bigr|^2\,\mathrm{d}\mu(t,x)\\
=& \sum_{|n|\geq N} |c_n|^2 + 
 \sum_{\substack{|m|,|n|\geq N\\ n\neq m}} c_n\overline{c_m}\int_{\R^2} e^{2\pi\ui \scal{\zeta,\phi(n)-\phi(m)}}
  \,\mathrm{d}\mu(\zeta) \\
  =& \sum_{|n|\geq N} |c_n|^2
  +\sum_{\substack{|m|,|n|\geq N\\ m\neq n}} c_n\overline{c_m}\, \widehat{\mu}\bigl(\phi(m)-\phi(n)\bigr).
  \end{aligned}
  \label{eq:1.5}
\end{equation}
Since $|\widehat{\mu}|\lesssim (1+|\xi|)^{-\delta}$ and
$|\phi(m)-\phi(n)|\approx |n-m|+||n|^s-|m|^s|$.  It follows that
\[
\begin{aligned}
  \bigl| \widehat{\mu}\bigl(\phi(m)-\phi(n)\bigr)\bigr|
  \lesssim&\frac{1}{\bigl(|n-m|+||n|^s-|m|^s|\bigr)^\delta}\\
  \lesssim&\begin{cases}
     \frac{1}{N^\delta}&\mbox{when }m=-n,\ |n|,|m|\geq N\\
    \dfrac{1}{\bigl| |n|^s-|m|^s \bigr|^\delta}&\mbox{when } |m|\neq |n|
\end{cases}.
\end{aligned}
\]
Choose $A_\mu>0$ such that $|\widehat{\mu}(\xi)|\le A_\mu(1+|\xi|)^{-\delta}$. We have $|\phi(m)-\phi(n)|\ge \bigl| |m|^s-|n|^s \bigr|$.
We then estimate
\begin{equation}
  \label{eq:thm1genmeas1}
\begin{aligned}
  \Bigl| \sum_{\substack{|m|,|n|\geq N\\ m\neq n}} c_n\overline{c_m}\, \widehat{\mu}\bigl(\phi(m)-\phi(n)\bigr) \Bigr|
  \le & \; \sum_{|n|\geq N}|c_nc_{-n}||\widehat{\mu}\bigl(\phi(-n)-\phi(n)\bigr)|\\
   &  + \; A_\mu\sum_{|n|\geq N}\sum_{\substack{|m|\geq N\\ |m|\neq |n|}}\frac{|c_nc_m|}{\bigl||n|^s-|m|^s\bigr|^{\delta}}:= S_1+ S_2.
\end{aligned}
\end{equation}

\proofskip The first sum $S_1$ is easy to estimate: writing
$|c_nc_{-n}|\leq\dfrac{|c_n|^2+|c_{-n}|^2}{2}$ and using the bound
$\bigl| \widehat{\mu}\bigl(\phi(-n)-\phi(n)\bigr) \bigr|\lesssim N^{-\delta}$ we
obtain
\begin{equation}
     S_1\lesssim  N^{-\delta}\sum_{|n|\geq N}|c_n|^2.
    \label{eq:thm1genmeas2}
\end{equation}
On the other hand, using the bound
\[
  \bigl| \widehat{\mu}\bigl(\phi(-n)-\phi(n)\bigr) \bigr|=\bigl|
  \widehat{\mu}(0,-2n) \bigr|\leq\eta
\]
we obtain $ S_1\leq \eta\dst \sum_{|n|\geq N}|c_n|^2$.  In particular
\begin{equation} \label{eq:thm1genmeas1b}
  \begin{aligned}
    (1-\eta)\sum_{|n|\geq N}|c_n|^2- S_2
    \leq & \;  \int_{\R^2} \Bigl| \sum_{|n|\geq N} c_ne^{2\ui\pi(|n|^st+nx)} \Bigr|^2\,\mathrm{d}\mu(t,x) \\
    \leq & (1+\eta)\sum_{|n|\geq N}|c_n|^2+ S_2.
  \end{aligned}
\end{equation}

\proofskip To bound $S_2$ we first write
\begin{equation*}
    \frac{|c_n||c_m|}{\bigl| |n|^s-|m|^s\bigr|^\delta}\le \frac{1}{2} \frac{|c_n|^2}{\bigl| |n|^s-|m|^s\bigr|^\delta}+ \frac{1}{2}\frac{|c_m|^2}{\bigl| |n|^s-|m|^s\bigr|^\delta}.
\end{equation*}
Substituting this into the sum we obtain
\begin{equation*}
    \begin{aligned}
    S_2\leq &\
\frac{A_\mu}{2}\sum_{|n|\ge N}\sum_{\substack{|m|\ge N\\|m|\neq |n|}}\frac{|c_n|^2}{\bigl| |n|^s-|m|^s \bigr|^\delta}
+\frac{A_\mu}{2}\sum_{|n|\ge N}\sum_{\substack{|m|\ge N\\|m|\neq |n|}}\frac{|c_m|^2}{\bigl| |n|^s-|m|^s \bigr|^\delta}\\
\leq &\, \frac{A_\mu}{2}\sum_{|n|\ge N}|c_n|^2\Bigl(\sum_{\substack{|m|\ge N\\|m|\neq |n|}}\frac{1}{\bigl| |n|^s-|m|^s \bigr|^\delta}\Bigr)
    +\frac{A_\mu}{2}\sum_{|m|\ge N}|c_m|^2\Bigl(\sum_{\substack{|n|\ge N\\|n|\neq |m|}}\frac{1}{\bigl| |n|^s-|m|^s \bigr|^\delta}\Bigr)
    \end{aligned}
\end{equation*}
with Fubini. Now, applying
Lemma~\ref{lem:sum-unified}~\ref{item:tail-decay}, as $s\delta>1$,
there exists $\sigma=\sigma(s,\delta)$ such that, for $N$ large enough,
\[
\sum_{\substack{|m|\ge N\\|m|\neq |n|}}\frac{1}{\bigl| |n|^s-|m|^s \bigr|^\delta}\lesssim\dfrac{1}{N^\sigma}.
\]
It follows that
$ S_2\dst\lesssim \dfrac{1}{N^\sigma}\sum_{|n|\ge N}|c_n|^2$. 
For $s$ large, choose $s_0>\max\left(1,1/\delta\right)$ and set
\[
K_{s_0,\delta}:=\sup_{n\in\Z}\sum_{\substack{m\in\Z\\|m|\ne|n|}}
\bigl||m|^{s_0}-|n|^{s_0}\bigr|^{-\delta}<\infty.
\]
For $s\ge s_0$ and integers $r>q\ge N\ge2$,
\[
r^s-q^s\ge q^{s-s_0}(r^{s_0}-q^{s_0})
\ge N^{s-s_0}(r^{s_0}-q^{s_0}).
\]
Consequently,
\[
S_2\le A_\mu K_{s_0,\delta}N^{-\delta(s-s_0)}
\sum_{|n|\ge N}|c_n|^2.
\]
For every $\eps>0$, one may choose
\[
\zeta(\mu,N)=s_0+
\frac{\max\left(0,\log(A_\mu K_{s_0,\delta}/\eps)\right)}{\delta\log N}
\]
so that $S_2\le\eps\sum_{|n|\ge N}|c_n|^2$  for any $s\ge \zeta(\mu,N)$.
Injecting the estimate \eqref{eq:thm1genmeas2} of $ S_1$ and this
estimate $ S_2$ into \eqref{eq:thm1genmeas1}, we obtain
\[
\Bigl| \sum_{\substack{|m|,|n|\geq N\\ n\neq m}} c_n\overline{c_m}\, \widehat{\mu}\bigl(\phi(n)-\phi(m)\bigr) \Bigr|
\lesssim \left(\dfrac{1}{N^\delta}+\dfrac{1}{N^\sigma}\right)\sum_{|n|\ge N}|c_n|^2
\leq \dfrac{1}{2}\sum_{|n|\ge N}|c_n|^2
\]
provided $N\geq N(\mu,s)$ for some $N(\mu,s)$ large enough depending only on the constant in the last $\lesssim$
therefore on $\mu,s$ only. Injecting this into \eqref{eq:1.5} gives \eqref{eq:thmhighfreqNbig}.

\proofskip To prove the second estimate, we need to be slightly more careful and use \eqref{eq:thm1genmeas1b}.
Note that this is an inequality with $\leq$'s not with $\lesssim$'s. It is then enough to
take $N\ge2$ and $s\ge\zeta(\mu,N,(1-\eta)/2)$ so that
$S_2\le\dfrac{1-\eta}{2}\sum_{|n|\ge N}|c_n|^2$.
Injecting this into \eqref{eq:thm1genmeas1b} gives \eqref{eq:thmhighfreqsbig}.
\end{proof}

We give two examples. The first one illustrates the need
for the condition $\delta s > 1$ in the theorem, whereas the second
recovers a well-known result.

\begin{example} \label{ex:highfreq-hyp-sharp}
Fix $s>1$ and $\delta>0$ with $\delta s\le 1$. Let $\mu$ be a
probability measure on $\R^2$ whose Fourier transform $\widehat{\mu}$
is real, positive, and such that
\[
\widehat\mu(\xi)=(1+|\xi|^2)^{-\delta/2}
\asymp(1+|\xi|)^{-\delta}.
\]
Such a probability measure exists since
\[
(1+|\xi|^2)^{-\delta/2}
=\frac{1}{\Gamma(\delta/2)}\int_0^\infty
r^{\delta/2-1}e^{-r}e^{-r|\xi|^2}\,\mathrm dr
\]
is a probability mixture of Fourier transforms of Gaussian measures.
Fix some
$N\ge 1$ and take the coefficient sequence
\[
c_n=\begin{cases} 1, & n=1,2,\dots,N,\\[2pt]
0,&\text{otherwise,}\end{cases}
\]
and define
\[
F_N(t,x):=\sum_{n=1}^N e^{2\ui\pi(|n|^s t + n x)}.
\]
Expanding the square and using the Fourier representation of \(\mu\)
then gives
\begin{align*}
\int_{\R^2} |F_N(t,x)|^2\,\mathrm d\mu(t,x)
= & \; \sum_{n=1}^N 1 + \sum_{\substack{1\le n\neq m\le N}} \widehat{\mu}\bigl(-(|n|^s-|m|^s)
,-(n-m)\bigr) =: N + \mathcal{S}_N.
\end{align*}
Let us assume \(1\le m<n\le N\). Then
$ |n|^s-|m|^s \;=\; s\,\xi^{\,s-1}\,(n-m)$ for some $\xi \in (m,n)$.
Moreover, since $s>1$,
\[
n^{s-1}(n-m)\le n^s-m^s\le s n^{s-1}(n-m),
\]
so that
\[
\bigl||n|^s-|m|^s\bigr| \asymp n^{\,s-1}\,(n-m).
\]
Using the decay of $\widehat{\mu}$ and restricting to $n/2 < m < n$ we
obtain
\[
\bigl|\widehat{\mu}\bigl(-(|n|^s-|m|^s),-(n-m)\bigr)\bigr|
\asymp \bigl(1 + n^{s-1}\,(n-m)\bigr)^{-\delta}
\asymp n^{-(s-1)\delta}\,(n-m)^{-\delta}.
\]
It follows that
\begin{align*}
\mathcal{S}_N
= & 2 \sum_{n=1}^N \sum_{m=1}^{n-1} \widehat{\mu}\bigl(-(|n|^s-|m|^s),-(n-m)\bigr)\\
\ge & \; 2 \sum_{n=1}^N \sum_{n/2 < m < n} \widehat{\mu}\bigl(-(|n|^s-|m|^s),-(n-m)\bigr)\\
\asymp & \; \sum_{n=1}^N n^{-(s-1)\delta}\Bigl(\sum_{1\le j<n/2}j^{-\delta}\Bigr).
\end{align*}
The same comparison in the full inner sum gives the matching upper bound.
Now use standard power-sum asymptotics for this sum.
Thus, for large \(N\), in the case that $\delta<1$ and $\delta s<1$, we find
\[
\mathcal{S}_N \asymp N^{\,2-\delta s}
\]
so that \(\mathcal{S}_N\) grows like \(N^{2-\delta s}\), which
dominates the diagonal term \(N\).
At the endpoint $\delta s=1$, for $1\le m<n$,
\[
1+(n^s-m^s)^2+(n-m)^2\le3n^{2s}.
\]
Hence, with $H_N=\sum_{n=1}^N n^{-1}$,
\[
\mathcal S_N\ge2\cdot 3^{-1/(2s)}
\sum_{n=2}^N\frac{n-1}{n}
=2\cdot 3^{-1/(2s)}(N-H_N),
\]
and therefore
\[
\liminf_{N\to\infty}\frac{1}{N}
\int_{\R^2}|F_N|^2\,\mathrm d\mu
\ge1+2\cdot3^{-1/(2s)}>\frac{3}{2}.
\]
These conclusions also hold above any fixed frequency cutoff $L\ge1$.
Indeed, for $N\ge L$, put
\[
F_{N,L}(t,x):=\sum_{n=L}^N e^{2\ui\pi(|n|^st+nx)}.
\]
Since $\mu$ is a probability measure, the reverse triangle inequality gives
\[
\|F_{N,L}\|_{L^2(\mu)}
\ge\|F_N\|_{L^2(\mu)}-(L-1).
\]
Dividing by $\sqrt{N-L+1}$ shows that the normalized energy still
diverges when $s\delta<1$, and satisfies the same lower bound $3/2$ for the limit inferior
when $s\delta=1$.
In particular, the condition
\(\delta s>1\) is sharp in
Theorem~\ref{thm:highfreqmeas} and at the endpoint $s\delta=1$ this disproves the upper constant $3/2$.
\end{example}

\begin{example}\label{ex:highfreq-open-bdd-set}
  Let $\omega\subset\R^2$ be a nonempty open bounded set. Let $\ffi$ be a
  $\cc^\infty$-smooth function supported on $\omega$ with $\ffi\geq 0$
  and $\dst\int_\omega\ffi(x)\d x=1$.  Then $\mu=\ffi\d t\d x$ satisfies
  the hypotheses of the theorem with $\delta$ arbitrarily large so
  that, if $s>1$ and if $N$ is large enough,
\begin{align*}
\norm{\ffi}_\infty\int_\omega\Bigl|\sum_{|n|\geq N}c_n e^{2\ui\pi(|n|^st+nx)} \Bigr|^2\d t\,\mathrm{d}x
\geq&\int_{\R^2} \Bigl|\sum_{|n|\geq N} c_ne^{2\ui\pi(|n|^st+nx)} \Bigr|^2\,\mathrm{d}\mu(t,x)\\
\geq&\dfrac{1}{2}\sum_{|n|\ge N}|c_n|^2.
\end{align*}
For the upper bound, we take $\psi$ a positive $\cc^\infty$-smooth
compactly supported function such that $\psi=1$ on $\omega$ and
$\lambda=\left(\int_{\R^2}\psi(x,y)\d x\,\mathrm{d}y\right)^{-1}$.  Then
$\mu=\lambda\psi\d t\d x$ satisfies the hypotheses of the theorem with
$\delta$ arbitrarily large so that, if $s>1$ and if $N$ is large
enough,
\begin{align*}
\int_\omega\Bigl| \sum_{|n|\geq N} e^{2\ui\pi(|n|^st+nx)} \Bigr|^2\d t\,\mathrm{d}x
\leq&\frac{1}{\lambda}\int_{\R^2} \Bigl| \sum_{|n|\geq N} e^{2\ui\pi(|n|^st+nx)} \Bigr|^2\,\mathrm{d}\mu(t,x)\\
\leq&\dfrac{3}{2}\left(\int_{\R^2}\psi(t,x)\d t\,\mathrm{d}x\right)\sum_{|k|\ge N}|c_n|^2.
\end{align*}
In summary, for every $s>1$, there is an $N(s,\omega)$ and a constant $C(s,\omega)$ such that
if $N\geq N(s,\omega)$, then for every $(c_n)_{n\in\Z}\subset\ell^2$,
\[
\frac{1}{C(s,\omega)}\sum_{|n|\ge N}|c_n|^2
\leq \int_\omega \Bigl| \sum_{|n|\geq N} c_ne^{2\ui\pi(|n|^st+nx)} \Bigr|^2\d t\,\mathrm{d}x
\leq C(s,\omega)\sum_{|n|\ge N}|c_n|^2.
\]
Note that when $\omega=[0,T]\times\tilde\omega$ with $\tilde\omega$ an open subset of $[0,1]$ much stronger
results are known, see \cite[Proposition 3.1]{BourgainBurqZworski}.
\end{example}

Now we draw a Corollary of Theorem~\ref{thm:highfreqmeas} whose first part 
is Theorem~\ref{thm:highfreq} from the introduction.

\begin{corollary}\label{cor:highfreq}
Let $T>0$ and $\Sigma$ be a compact smooth curve in $[0,T]\times\T$ with non-zero curvature
and $\sigma$ be the normalized arc length measure on $\Sigma$. 
Then, for every $s>2$, there exists $N(s,\Sigma)$ such that, if $N\geq N(s,\Sigma)$, for every $(c_n)_{n\in\Z}\in\ell^2$,
\begin{equation}
    \label{eq:corhighfreqNbigcurve}
\dfrac{1}{2}\sum_{|n|\ge N}|c_n|^2\leq  \int_{\Sigma} \Bigl| \sum_{|n|\geq N}c_n e^{2\ui\pi(|n|^st+nx)} \Bigr|^2\,\mathrm{d}\sigma(t,x) 
\leq \dfrac{3}{2}\sum_{|n|\ge N}|c_n|^2.
\end{equation}
and, when $N\geq2$, there exists $\zeta(N,\Sigma)>0$ and $C=C(N,\Sigma)>0$ such that, if $s\geq\zeta(N,\Sigma)$, then for every $(c_n)_{n\in\Z}\in\ell^2$,
\begin{equation}
    \label{eq:corhighfreqsbigcurve}
\dfrac{1}{C}\sum_{|n|\ge N}|c_n|^2\leq  \int_{\Sigma} \Bigl| \sum_{|n|\geq N} c_ne^{2\ui\pi(nx+|n|^st)} \Bigr|^2\,\mathrm{d}\sigma(t,x) 
\leq C\sum_{|n|\ge N}|c_n|^2.
\end{equation}
\end{corollary}

\begin{proof}
  Let $\mu$ be the normalized arc-length measure on a compact curve in
  $\R^2$ with non-vanishing curvature.  It is a classical consequence
  of stationary phase (see e.g. Stein\cite[Chap.~VIII]{St93}
  or \cite[Chapter 2]{Grafakos})
  that then $\widehat{\mu}(\xi)=O(|\xi|^{-\onehalf })$.  Therefore, the
  hypothesis of Theorem~\ref{thm:highfreqmeas} are satisfied with
  $\delta=\dfrac{1}{2}$, which yields the condition $s>2$ (recall
  Example~\ref{ex:highfreq-hyp-sharp}).  
\end{proof}

Although the curvature of $\Gamma_p$ may degenerate at $t=0$, this causes no
difficulty for the lower estimate: for every $T>0$ we apply Corollary \ref{cor:highfreq} to the subarc
$\{(t,p(t)):T/2\leq t\leq T\}$, on which the curvature is bounded away from
zero, and then use positivity of the integral. Since $\mbox{d}\sigma$ and $\mbox{d}t$ are equivalent on this subarc, 
and by positivity of the integral, the claim follows. The upper bound on $[0,T]$ follows from the upper estimate in Theorem~\ref{thm:schroedinger} since $s>2$. This argument applies for every $\alpha>1$. Accordingly, in Corollary \ref{cor:highfreqintro} one may replace $\alpha\ge 2$ by $\alpha>1$ and write its cutoff as $N(s,T,p)$.

Finally, we would like to stress that there are also fractal measures
that have Fourier polynomial decay, see for instance
\cite{BakerBanaji} or \cite{LiSahlsten} and references therein for
such measures based on iterated function systems.

\subsection{Strong Ingham Inequality for highly dispersive sums}\label{Sec:highdisp}

In this section, we are going to improve on Theorem~\ref{thm:highfreq}
by showing that the restriction $N\geq N(s)$ can be removed when $s$
is large enough. We first state an auxiliary result.

\begin{lemma}\label{lem:simple33}
  Let $\Gamma$ be a nonconstant $\cc^2$-curve   in the $(t,x)$--plane and let
  $c_{-1},c_0,c_1\in\C$.  Then
\[
  \sum_{n=-1}^1 c_n e^{2\ui\pi(|n|t+nx)}=0 \qquad\text{on }\Gamma
\]
if and only if one of the following holds:
\begin{enumerate}
\renewcommand{\theenumi}{\roman{enumi}}
  \item $\Gamma$ is a horizontal segment $\{(t,x_0):t\in I\}$, $c_0=0$ and 
        $c_{-1}e^{-2\ui\pi x_0}+c_1 e^{2\ui\pi x_0}=0$;
  \item $\Gamma=\{(\beta-x,x)\}$ is contained in the line of slope $-1$, $c_{-1}=0$, and 
        $c_{1}e^{2\ui\pi\beta}+c_0=0$;
  \item $\Gamma=\{(\beta+x,x)\}$ is contained in the line of slope $+1$, $c_{1}=0$ and 
        $c_0+c_{-1}e^{2\ui\pi\beta}=0$;
  \item $c_{-1}=c_0=c_1=0$.
\end{enumerate}
\end{lemma}

\begin{proof}
If $\Gamma$ is horizontal, say $(t,x_0)$, then the identity becomes
\[
c_0+ c_{-1}e^{2\ui\pi(t-x_0)} + c_1 e^{2\ui\pi(t+x_0)}=0.
\]
Linear independence of $1,e^{2\ui\pi t}$ yields $c_0=0$ and the stated relation between
$c_{-1}$ and $c_1$.

\smallskip\noindent Assume from now on that $\Gamma$ is not
horizontal.  Then we may parametrize $\Gamma$ as $(t,x)=(\gamma(x),x)$
with $\gamma'$ nowhere zero.  Define
\[
F(x)=\sum_{n=-1}^1 c_n e^{2\ui\pi(|n|\gamma(x)+nx)}.
\]
We are assuming that $F\equiv0$ on an interval.

\smallskip

Consider the three functions
\[
f_{-}(x)=e^{2\ui\pi(\gamma(x)-x)},\qquad
f_0(x)=1,\qquad
f_{+}(x)=e^{2\ui\pi(\gamma(x)+x)}.
\]
They satisfy a linear ODE of the form
\begin{align*}
  f_\pm'(x)    = & 2\ui\pi(\gamma'(x) \pm 1) f_\pm(x)\\
  f_\pm''(x)   = & \bigl( 2\ui\pi \gamma''(x) - 4\pi^2 (\gamma'(x) \pm 1)^2 \bigr) f_\pm
\end{align*}
Their Wronskian $W(x)=\det(f_j^{(k)}(x))_{|j|\le 1, \, k=0,1,2}$ can be computed
explicitly and factorizes as
\[
   W(x)=  -8 \pi^2 \bigl( \gamma''(x) - 2\ui\pi (\gamma'(x)^2 -1) \bigr) f_+ f_-.
\]
$W(x)\equiv 0$ if and only if $\gamma''(x) - 2\ui\pi (\gamma'(x)^2 -1) = 0$.
Considering real and imaginary parts, $\gamma''(x)=0$ and $\gamma'(x) = \pm 1$,
thus $\gamma(x) = \pm x + \gamma_0$. The remaider proof is straightforward.
\end{proof}

\begin{remark}
\label{rem:few_points}
Whenever the three restricted exponentials are linearly independent on
$\Gamma$, three suitable points of $\Gamma$ give an invertible
evaluation matrix. Indeed, their evaluation vectors span $\C^3$,
otherwise a nonzero linear functional would yield a dependence on the
entire curve. See Section~\ref{sec:few-points} in the appendix for a sufficient criterion.
\end{remark}

\begin{corollary}\label{cor:simple33}
  Let $s,\alpha>1$ and $p\in \cc^1([0,\infty[)\cap \cc^2(]0,\infty[)$
  satisfying Assumption~\ref{ass:H-alpha}.  Then, for every $T>0$,
  there exists $C=C(p,T)>0$ such that, for every $c_{-1},c_0,c_1\in\C$
\begin{equation}
    \label{eq:corsimple}
\dfrac{1}{C}\sum_{n=-1}^1|c_n|^2\leq 
\int_0^T\abs{\sum_{n=-1}^1c_ne^{2\ui\pi\bigl(|n|^st+np(t)\bigr)}}^2\,\mathrm{d}t
\leq C\sum_{n=-1}^1|c_n|^2.
\end{equation}
\end{corollary}

\begin{proof}[Proof of Corollary~\ref{cor:simple33}]
  First note that, when $n=-1,0$ or $1$, $|n|^s=|n|$ so that the
  integral in \eqref{eq:corsimple} does not depend on $s$. Next, the
  curve $\bigl\{\bigl(t,p(t)\bigr), t\in[0,T]\bigr\}$ is not a line
  segment so that, due to Lemma~\ref{lem:simple33},
  $\dst\sum_{n=-1}^1c_ne^{2\ui\pi\bigl(|n|^st+np(t)\bigr)}$ can not be
  zero on all of $[0,T]$. It follows that
\[
\dst (c_{-1},c_0;c_1)\to \int_0^T\abs{\sum_{n=-1}^1c_ne^{2\ui\pi\bigl(|n|^st+np(t)\bigr)}}^2\,\mathrm{d}t
\]
is a norm on $\C^3$ so that it is equivalent to the $\ell^2$-norm.
\end{proof}

\begin{theorem}\label{th:merge}
  Let $\alpha>1$ and $p\in \cc^1([0,\infty[)\cap \cc^2(]0,\infty[)$
  satisfying Assumption~\ref{ass:H-alpha}.  Then, for every $T>0$,
  there exists $s_T>1$ and $C_T>0$ such that, if $s>s_T$, then for
  every $(c_n)_{n\in\Z}\in\ell^2$
\[
\frac{1}{C_T}\sum_{n\in\Z}|c_n|^2\leq \int_0^T\abs{\sum_{n\in\Z}c_ne^{2\ui\pi\bigl(np(t)+|n|^st\bigr)}}^2\,\mathrm{d}t
\leq C_T\sum_{n\in\Z}|c_n|^2.
\]
\end{theorem}

\begin{proof}
  The upper bound can be proved in Theorem~\ref{thm:schroedinger} by noting that it is uniform in $s$. Hence we only
  prove the lower bound.  By the Corollary~\ref{cor:highfreq} on high
  frequencies, there exist constants $\zeta>0$ and $C_+$, depending
  only on $p,\alpha,T$,  such that for every $s\ge\zeta$ and every
  sequence $(c_n)_{|n|\ge2}\in\ell^2$
  \[
    \int_0^T\Big|\sum_{|n|\ge2} c_n e^{2 \ui \pi (n p(t)+|n|^s t)}\Big|^2 \,dt
    \ge \frac{1}{C_+}\sum_{|n|\ge2}|c_n|^2.
  \]
  Also, by Corollary~\ref{cor:simple33} there is $C_-$, depending only
  on $p,T$, such that for every triple $(c_{-1},c_0,c_1)$
  \[
    \int_0^T\Big|\sum_{|n|\le 1} c_n e^{2\ui\pi  (n\, p(t)+|n|^s t)}\Big|^2 \,dt
    \ge \frac{1}{C_-}\sum_{|n|\le 1} |c_n|^2.
  \]
  Let $\phi_n(t):=e^{2 \ui \pi (n\, p(t)+|n|^s t)}$, and write
  \[
    u_-(t):=\sum_{|n|\le 1} c_n\phi_n(t),\qquad
    u_+(t):=\sum_{|n|\ge 2} c_n\phi_n(t),\qquad u=u_-+u_+.
  \]
  Combining the two lower bounds above we get
  \begin{align*}
    \int_0^T |u|^2 \,dt
    = & \; \int_0^T|u_-(t)|^2 + |u_+(t)|^2 + 2\Re(u_-(t)\overline{u_+}(t))\,dt \\
  \ge & \; \frac{1}{C_-}\sum_{|n|\le 1}|c_n|^2
      +\frac{1}{C_+}\sum_{|n|\ge 2}|c_n|^2 - 2\Big|\int_0^T u_-(t)\overline{u_+}(t) \,dt\Big|.
  \end{align*}
  We estimate the cross term by Cauchy-Schwarz and the oscillatory
  integrals from Proposition~\ref{prop:vdc}. To stress the dependence
  on $s$, and the fact that $T$ is fixed, we here denote them by
  $I_{m,n}^{(s)}$.
    \begin{equation}    \label{eq:absorb-high-disp}
  \begin{aligned}
        & \; \Big|\int_0^T u_-(t)\overline{u_+}(t)\,dt\Big|
    \; \le \;  \sum_{|m| \le 1}\sum_{|n|\ge2} |c_m||c_n| \; I_{m,n}^{(s)} \\
    \le & \; \frac12\Big(\sup_{|m| \le 1}\sum_{|n|\ge2} I_{m,n}^{(s)}\Big)\sum_{|m|\le1}|c_m|^2
    +\frac12\Big(\sup_{|n|\ge2}\sum_{|m|\le 1} I_{m,n}^{(s)}\Big)\sum_{|n|\ge2}|c_n|^2 .
  \end{aligned}
  \end{equation}
  Fix $T>0$ and set $\tau=2c_2T^{\alpha-1}$, as in Proposition~\ref{prop:vdc}.  
  For each fixed $|m| \le 1$ and each $|n|\ge2$
  the ratio $R_{m,n}(s)=(|n|^s-|m|^s)/(n-m)$ tends to $+\infty$ when $n\ge2$
  and to $-\infty$ when $n\le-2$ as $s\to\infty$. Hence there exists
  $s_T$ such that for each $s\ge s_T$, every pair
  $(m,n)$ with $|m|\le1$, $|n|\ge2$ belongs to the set $S_{\mathrm{good}}$.
  By Proposition~\ref{prop:vdc} (case (i)) we thus have for those pairs
  the uniform bound
  \[
    I_{m,n}^{(s)}\le C'_T\frac{1}{\big||n|^s-|m|^s\big|},
  \]
  where the implied constant $C'_T$ depends only on $T$ and $p$, but  not on $s$ for $s\ge s_T$.

  \proofskip Using $|n|^s-|m|^s\ge |n|^s - 1$ for $|n|\ge2$, $|m|\le 1$ and $s\ge s_T\ge 2$
  yields $I_{m,n}^{(s)} \lesssim \frac{1}{n^2-1}$. This ensures summability over
  $n$, and via a simple dominated convergence argument,
  \[
    \sup_{|m| \le 1} \sum_{|n|\ge2} I_{m,n}^{(s)}\xrightarrow[s\to\infty]{} 0,
    \qquad
    \sup_{|n|\ge2}\sum_{|m|\le 1} I_{m,n}^{(s)}\xrightarrow[s\to\infty]{} 0.
  \]
  Therefore we can choose $s_T$ so large that both suprema are below
  $1/(2(C_-+C_+))$.  For such $s$, inequality
  \eqref{eq:absorb-high-disp} gives
  $\big|\int_0^T u_-(t)\overline{u_+}(t)\,dt\big|\le
  \frac{1}{4(C_-+C_+)}\sum_{n}|c_n|^2$.  Plugging into the lower bound
  for $\int_0^T|u|^2$ and absorbing constants yields
  \[
    \int_0^T |u(t)|^2\,dt \ge \frac{1}{C_T}\sum_{n\in\Z}|c_n|^2
  \]
  for some $C_T$ depending only on $p,\alpha,T$.  This proves the claim.
\end{proof}

\section{Zero sets of dispersive trigonometric sums}\label{Sec:lowfreq}

Let $s>0$ be fixed and $N \ge 1$ be a natural number. Let
$(\lambda_n)_{n=-N\ldots N}$ be a finite {\it real} sequence of {\em distinct} ($\lambda_n \not=\lambda_m$ when $n\not=m$) numbers that will be called frequencies. Let
\begin{equation}  \label{eq:low-frequency-F}
  F(t,x) := \sum_{n=-N}^N  c_n \,  e^{2\ui\pi\, (|n|^s t + \lambda_n x)},
\end{equation}
where $(c_n)_{n=-N,\ldots,N}\subset\C$ is a finite sequence of complex coefficients.

In this section, we are investigating curves $\Gamma=\bigl\{\bigl(t,\gamma(t)\bigr),\ t\in[0,T]\bigr\}$
on which a function $F$ of the form \eqref{eq:low-frequency-F} can {\em not} vanish.
As for Corollary~\ref{cor:simple33}, if $\Gamma$ is such a curve then, for every $T>0$
\[
(c_n)_{n=-N,\ldots,N}\to \int_0^T\abs{\sum_{n=-N}^N  c_ne^{2\ui\pi\, \bigl(|n|^s t + \lambda_n \gamma(t)\bigr)}}^2\,\mathrm{d}t
\]
is the square of a norm on $\C^{2N+1}$. Therefore, it is equivalent to the $\ell^2$ norm, {\it i.e.} there exists
$C=C(N,\Gamma)$ such that
\begin{equation}     \label{eq:equivnorm}
\dfrac{1}{C}\sum_{n=-N}^N|c_n|^2\leq
\int_0^T\abs{\sum_{n=-N}^N  c_ne^{2\ui\pi \bigl(|n|^s t + \lambda_n \gamma(t)\bigr)}}^2\,\mathrm{d}t
\leq C\sum_{n=-N}^N|c_n|^2.
\end{equation}
Results in this section are summarized in
Theorem~\ref{thm:lowfreq-intro} in the particular case of
$\lambda_n=n$.  The first case in that theorem is a direct consequence
of the following lemma:

\begin{lemma}\label{lem:observation-curves-are-holomorphic}
Let $F$ be a function of the form \eqref{eq:low-frequency-F}.
  Assume that there is some continuous function $\gamma: [a,b] \to \C$
  satisfies that $F\bigl(t, \gamma(t)\bigr) = 0$ and $\partial_x F(t,\gamma(t))\neq 0$ for every $t\in (a,b)$. Then 
  $\gamma$ extends holomorphically to a complex neighborhood of
  $(a,b)$.
\end{lemma}
\begin{proof}
   Since $F$ is entire on $\C^2$, then at each $t_0\in (a,b)$, the holomorphic implicit function theorem gives a unique holomorphic zero graph in a product neighborhood of $(t_0,\gamma(t_0))$. By continuity, the given graph lies in the neighborhood for real $t$ sufficiently close to $t_0$, and therefore agrees with this holomorphic graph. Choose its extension on a disc centered at $t_0$ whose real diameter is contained in $(a,b)$ and on which this agreement holds. If two such discs intersect, their real diameters overlap on an open interval. Since their holomorphic graphs agree on the interval, they agree on the whole intersection. The extensions thus glue on the union of the discs, which forms a complex neighborhood of $(a,b)$.
\end{proof}

The above lemma tells us that only curves parameterized by an
holomorphic function needs to be considered when investigating an
equation of the form 
\begin{equation*}
    F\bigl(t, \gamma(t)\bigr)=0, \quad \partial_x F(t,\gamma(t))\neq 0 \quad \forall t\in (0,T).
  \end{equation*}
  
  One may for instance consider $\gamma$ that satisfies
  Assumption~\ref{ass:H-alpha} and is not an analytic function over
  $(0,T)$.  Although this assumption applies to a broad class of
  functions, it excludes some simple examples such as polynomials, and
  power functions $p(t)=t^\alpha$.  The remaider of this section is
  devoted to this issue. More precisely, we prove the following:

\begin{theorem}\label{thm:merhol}
Let $s > 0$, $N\ge 0$, and let $(\lambda_n)_{n=-N}^N\subset\R$ be pairwise
distinct frequencies and $(c_n)_{n=-N}^N\subset\C$. Define
\[
F(t,x)=\sum_{n=-N}^N c_n\,e^{2\ui\pi(|n|^s t+\lambda_n x)}, t,x\in\C.
\]
Let $T>0$ and let $\gamma$ be a continuous function that satisfies one of the following assumptions:
\begin{enumerate}
\renewcommand{\theenumi}{\roman{enumi}}
\item there exists an open connected neighborhood $\Omega$ of the real interval $[0,T]$
and $a\in\Omega\setminus\R$ such that $\gamma$ is meromorphic over $\Omega\setminus[0,T]$
with a pole at $a$;

\item or $\gamma$ is an entire function, but not a polynomial of degree $\leq 1$.
\end{enumerate}
If
\[
  \{ t \in (0, T):    F\bigl(t,\gamma(t)\bigr)=0 \}
\]
has an accumulation point, then $F\equiv 0$, {\it i.e.} $c_n=0$ for all $n$.
\end{theorem}


\noindent In the case of affine curves $\gamma(t) = \alpha  + \beta t$, restriction reduces to the question of linear independence of the functions
\[ 
  t \mapsto e^{2 \ui \pi (|n|^s+\beta\lambda_n) t }
\]
which in turn reduces to whether the frequencies   $(|n|^s+\beta \lambda_n)_{n=-N, \ldots, N}$ are all distinct.

\noindent Before passing on to the proof, we draw a consequence of the
norm equivalence argument \eqref{eq:equivnorm} above.

\begin{corollary}
Taking the notation of Theorem \ref{thm:merhol} and taking further $0 \not=\ffi\in L^1([0,T])$
  a non-negative function. 
Then, for every $N\geq 0$ there exists a constant
  $C=C(N,\gamma,\ffi,\Omega)\geq 1$ such that, for every
  $(c_n)_{|n|\leq N}$, we have
\[
    \frac{1}{C}\int_\Omega \abs{ F(t, x) }^2\,\mathrm{d}t\,\mathrm{d}x
\leq  \; \int_0^T \abs{ F(t, \gamma(t)) }^2\,\ffi(t)\,\mathrm{d}t
\leq  C\int_\Omega|F(t,x)|^2 \,\mathrm{d}t\,\mathrm{d}x
\]
and
\[
\frac{1}{C}\sum_{n=-N}^N|c_n|^2\leq \int_0^T\abs{
\sum_{n=-N}^N  c_n \,  e^{2\ui\pi \bigl(|n|^s t + \lambda_n \gamma(t)\bigr)}}^2\,\ffi(t)\,\mathrm{d}t
  \leq C\sum_{n=-N}^N|c_n|^2.
\]
\end{corollary}

\begin{proof}
It is enough to notice that 
\[
(c_k)_{k\in\{-N,\ldots,N\}}\to
\int_\Omega\abs{\sum_{k=-N}^N  c_k \,  e^{2\ui\pi (|k|^s t + \lambda_k x)}}^2\d x\,\mathrm{d}t
\]
and
\[
(c_k)_{k\in\{-N,\ldots,N\}}\to\int_0^T\abs{\sum_{k=-N}^N  c_k \,  e^{2 i\pi  (|k|^s t + \lambda_k \gamma(t))}}^2\, \ffi(t)\,\mathrm{d}(t)
\]
are the squares of norms on $\C^{2N+1}$ and are thus equivalent (the constant of
norm equivalence however depends on dimension, so on
$N$). Sub-additivity and positive homogeneity are both obvious.  The
only thing that needs to be noticed is that both only vanish when
$c=0$. For the first one, this is a direct consequence of linear
independence of
$\{(t, x)\mapsto e^{2\ui\pi (|k|^\alpha t + \lambda_k x)}\}$.  For the
second one, this is a direct consequence of Theorem~\ref{thm:merhol},
and the fact that any set of positive measure has an accumulation
point.
\end{proof}

We can choose
$\ffi(t)\,\mathrm{d}t=\sqrt{1+|\gamma'(t)|^2}\,\mathrm{d}t$, the arc length
on the curve $\{\bigl(t,\gamma(t)\bigr),\ t\in[0,T]\}$ or $\ffi=1$.  The
particular case $\lambda_n=n$ is then Theorem~\ref{thm:lowfreq-intro}.

\begin{proof}[Proof of Theorem~\ref{thm:merhol}] 
Let us first assume that there exists an open connected neighborhood $\Omega$ of the real interval $[0,T]$
and $a\in\Omega\setminus\R$ such that $\gamma$ is meromorphic over $\Omega$
with a pole at $a$. As $F$ is holomorphic in $t,x$, by analytic continuation
\[
  \{ t \in (0, T):    F\bigl(t,\gamma(t)\bigr)=0 \}
\]
has an accumulation point if and only if 
\[
   F\bigl(z,\gamma(z)\bigr)=0,\ \forall z\in\Omega
\]
that is
\begin{equation}
\label{eq:newmerhol}
\sum_{n=-N}^N c_n\,e^{2\ui\pi \bigl(|n|^s z+\lambda_n \gamma(z)\bigr)}=0,\ \forall z\in\Omega.
\end{equation}
We now distinguish two cases

\proofskip

\noindent{\bf Case 1:} {\em $\gamma$ is a pole of finite order $m\geq 1$ at $a$}.

\proofskip In this case $(z-a)^m\gamma(z)=\tilde\gamma(z)$ is
holomorphic with $\tilde\gamma(a)\not=0$.  Write
$\tilde\gamma(a)=\rho e^{\ui \theta}$ and consider the sequence
$z_k=a+\dfrac{1}{k}\exp\left(\dfrac{i}{m}\left(\theta+\dfrac{\pi}{2}\right)\right)$
so that $\gamma(z_k)\sim -i\rho k^m$.  Assume towards a contradiction
that the $c_n$'s are not all $0$ and let $n_0$ be the index $n$ for
which $\lambda_n$ is the largest among those $n's$ for which
$c_n\not=0$.  Formally, let $S=\{n\,:c_n\not=0\}$ and
$n_0=\mathrm{argmax}_{n\in S}\lambda_n$. As all the $\lambda_n's$ are
distinct, this is well defined. Take $z=z_k$ in \eqref{eq:newmerhol}
and factor out $e^{2\ui\pi\lambda_{n_0}\gamma(z_k)}$ to obtain
\[
0= \; c_{n_0}e^{2\ui\pi|n_0|^sz_k}+\sum_{n\in S\setminus\{n_0\}} c_n\,e^{2\ui\pi |n|^s z_k} \,
e^{2\ui\pi \bigl(\lambda_n-\lambda_{n_0}\bigr) \gamma(z_k)}
\underset{k\to\infty}{\relbar\joinrel\relbar\joinrel\relbar\joinrel\to}  \; c_{n_0}e^{2\ui\pi|n_0|^sa}
\]
since $z_k\to a$, and
$\Re\left( \ui\bigl(\lambda_n-\lambda_{n_0}\bigr) \gamma(z_k) \right)
\sim \rho \bigl(\lambda_n-\lambda_{n_0}\bigr)k^m\to-\infty$ as
$\lambda_{n_0}>\lambda_n$. It follows that $c_{n_0}=0$, a
contradiction.

\proofskip

\noindent{\bf Case 2:} {\em $\gamma$ has an essential singularity at $a$}. 

\proofskip In this case, we construct the sequence $z_k$ with the help
of the Great Picard Theorem.  Indeed, it is enough to obtain a
sequence $z_k\to a$ such that $i\gamma(z_k)\in\R^-$ and
$i\gamma(z_k)\to-\infty$.  As all complex numbers, except possibly
one, are taken by $\gamma$ (infinitely often) in any neighborhood of
$a$, such a sequence $z_k$ exists. One can then argue in exactly the
same way as previously.

\proofskip 
Assume now that $\gamma$ is entire and non-affine. There exists
a sequence $z_k$, with $|z_k|\to\infty$, such that
\[
\frac{-\Im\gamma(z_k)}{|z_k|}\longrightarrow+\infty.
\]
Indeed, otherwise $\Re(i\gamma(z))\le C(1+|z|)$ on $\C$.
The Borel--Carathéodory inequality would then give
$\max_{|z|\le R}|\gamma(z)|=O(R)$, and Cauchy's estimates
would imply that $\gamma$ is affine, hence we obtain a contradiction.

By the identity theorem, $F(z,\gamma(z))=0$ for every $z\in\C$.
Suppose that $S:=\{n:c_n\neq0\}$ is nonempty, and choose
$n_0\in S$ such that $\lambda_{n_0}=\max_{n\in S}\lambda_n$.
Dividing by the corresponding exponential gives
\[
0=c_{n_0}+
\sum_{n\in S\setminus\{n_0\}}c_n
e^{2\pi i[(|n|^s-|n_0|^s)z+
(\lambda_n-\lambda_{n_0})\gamma(z)]}.
\]
For each $n\in S\setminus\{n_0\}$, the logarithm of the
modulus of its exponential at $z=z_k$ is bounded above by
\[
2\pi\bigl||n|^s-|n_0|^s\bigr|\,|z_k|
-2\pi(\lambda_{n_0}-\lambda_n)
\bigl[-\Im\gamma(z_k)\bigr]\longrightarrow-\infty.
\]
Letting $k\to\infty$ therefore yields $c_{n_0}=0$,
a contradiction. Thus all coefficients vanish.
\end{proof}

\section{Schrödinger equations with bounded potentials}\label{Sec:schroedinger-potential}

As an application of our results, we now show that solutions of fractional Schrödinger equations with
small bounded potentials $V\in L^\infty(\T)$
\begin{equation}\label{eq:schroedinger-potential}
\begin{cases}
    i\partial_tu=(-|\partial_x|^s+V)u, \quad x\in \T,t\in \R,\\
    u(0,x)=u_0   
\end{cases}
\end{equation}
and initial data $u_0\in L^2(\T)$ have traces along curves satisfying Assumption~\ref{ass:H-alpha}.
Moreover, the $L^2$ norm of those traces bound the $L^2$ norm of the initial value $u_0$.
This can be seen as an observability inequality over a curve, thus a set of measure $0$. 

\begin{theorem}\label{th:schpot}
  Let $s$ and $\alpha$ be the same as in
  Theorem~\ref{thm:schroedinger},
  $p\in \cc^1([0,\infty[)\cap \cc^2(]0,\infty[)$ satisfy
  Assumption~\ref{ass:H-alpha}. Let $V\in L^\infty(\T;\R)$. Then, for
  every $T>0$ there exists $C(T,\|V\|_\infty)$ such that, for every
  solution $u$ of \eqref{eq:schroedinger-potential}, we have
\begin{equation*}
    \int_0^T \abs{u\bigl(t,p(t)\bigr)}^2\,\mathrm{d} t\le C(T,\|V\|_\infty) \;  \|u\|^2_{L^2(\T)}.
\end{equation*}
Further, assume that $s>\max \left(\dfrac{3}{2}, 2-\dfrac{1}{\alpha}\right)$. There exists $T(p)>0$, and $\varepsilon_0(T)>0$ such that for every $T> T(p)$, every potential satisfying $\|V\|_\infty\le \varepsilon_0(T)$, there exists a positive constant $C(p,\|V\|_\infty)>0$ such that for every solution $u$ of \eqref{eq:schroedinger-potential}, we have
\begin{equation*}
C(p,\|V\|_\infty) \|u\|^2_{L^2(\T)}\leq 
\frac{1}{T}\int_0^T\abs{u\bigl(t,p(t)\bigr)}^2\,\mathrm{d} t.
\end{equation*}
\end{theorem}

\begin{proof}
Given any $u_0\in L^2(\T)$, we write the solution $u$ by Duhamel's formula
\begin{equation*}
    u(t,x)=e^{\ui t(|\partial_x|^s-V)}u_0=e^{\ui t|\partial_x|^s}u_0-i\int_0^t e^{\ui (t-\tau)|\partial_x|^s} V(x)u(\tau,x)\d \tau.
\end{equation*}
One can write it as
\begin{align}
    u(t,x)=&\,e^{\ui t|\partial_x|^s}u_0-i\int_0^t \mathbf{1}_{\{\tau<t\}} e^{\ui (t-\tau)|\partial_x|^s} V(x)e^{\ui \tau(|\partial_x|^s-V)}u_0\d \tau\notag\\
    =&\, e^{\ui t|\partial_x|^s}u_0-i\int_0^t \mathbf{1}_{\{\tau<t\}} e^{\ui (t-\tau)|\partial_x|^s} w_\tau\d \tau\label{eq:duhamel}
    \end{align}
where we denote by $w_\tau:=w_\tau (x):=V(x)e^{\ui \tau(|\partial_x|^s-V)}u_0=V(x)u(\tau,x)$ in the second line. Notice that $w_\tau(x)\in L_\tau^\infty L_x^2:=L^\infty_\tau((0,\infty);L^2_x(\T))$. For simplicity, we define for any $T_2\ge T_1\ge 0$ the spacetime curve
\begin{equation*}
    \Gamma_p^{T_1,T_2}:=\lbrace (t,p(t))\in \R\times \T: T_1\le t\le T_2 \rbrace.
\end{equation*}
When $T_1=0$, we simply write $\Gamma_p^T=\Gamma_p^{0,T}$ and then introduce
\begin{equation*}
    \|v\|_{\Gamma_p^{T_1,T_2}}:=\biggl( \int_{T_1}^{T_2} |v(t,p(t))|^2\d t \biggr)^{\frac{1}{2}}
\end{equation*}
whenever the integral in the right-hand side is meaningful.

For the first claim, Duhamel's formula \eqref{eq:duhamel} further gives
\begin{equation*}
    |u(t,x)|^2\le 2 \bigl| e^{\ui t|\partial_x|^s} u_0\bigr|^2+2 \biggl|\int_0^t \mathbf{1}_{\tau<t} e^{\ui (t-\tau)|\partial_x|^s} w_\tau\d \tau\biggr|^2
\end{equation*}
Then we integrate it along the curve $\Gamma_p^{0,T}$ and obtain
\begin{equation}\label{eq:integral-along-curve}
    \|u\|^2_{\Gamma_p^{T}}\le  2\|e^{\ui t|\partial_x|^s}u_0 \|_{\Gamma_p^{0,T}}^2+2  \biggl\|\int_0^\infty\mathbf{1}_{\tau<t} w_\tau (t-\tau,x)\d \tau\biggr\|^2_{\Gamma_p^{0,T}}.
\end{equation}
where we used the notation $w_\tau (t,x):=e^{\ui t|\partial_x|^s}w_\tau$. For the first term in the right-hand side of \eqref{eq:integral-along-curve}, we use \eqref{eq:obs-curve-schroedinger-upper} in Theorem~\ref{thm:schroedinger} and then obtain
\begin{equation}\label{eq:upperbound-first}
    \|e^{\ui t|\partial_x|^s}u_0 \|_{\Gamma_p^{T}}^2\le C(T)\|u_0\|^2_{L^2(\T)}.
\end{equation}
For the second term in the right-hand side of \eqref{eq:integral-along-curve}, by the Cauchy--Schwarz, we have
\begin{equation*}
    \begin{aligned}
\left\|\int_0^\infty\mathbf{1}_{\{\tau<t\}} w_\tau (t-\tau,x)\d \tau\right\|^2_{\Gamma_p^{T}}
=&\, \int_0^T \left| \int_0^\infty \mathbf{1}_{\{\tau<t\}}w_\tau(t-\tau,p(t))\d \tau\right|^2\d t\\
 \le &\, \int_0^Tt\int_0^\infty \bigl|\mathbf{1}_{\{\tau<t\}}w_\tau(t-\tau,p(t))\bigr|^2\d \tau \d t\\
         \le &\, T\int_0^T \int_0^\infty \bigl|\mathbf{1}_{\{\tau<t\}}w_\tau(t-\tau,p(t))\bigr|^2\d \tau \d t\\
         =&\, T\int_0^\infty\int_0^{T-\tau} \bigl|\mathbf{1}_{\{t>0\}}w_\tau(t,p(t+\tau))\bigr|^2\d t\mathrm{d}\tau\\
         =&\, T\int_0^\infty \|e^{\ui t|\partial_x|^s}w_\tau(x)\|_{\Gamma_{q_\tau}^{T-\tau}}^2\d \tau
    \end{aligned}
\end{equation*}
where we defined $q_\tau:=q_\tau(t):=p(t+\tau)$.  Here we notice that the curve $\Gamma_{q_\tau}^{0,T-\tau}$ is just the parallel translation of $\Gamma_{p}^{\tau, T}$ with the latter as a part of $\Gamma_{p}^{0, T}$, see Fig.~\ref{fig:parallel-translation}. This translation  does not influence \eqref{eq:obs-curve-schroedinger-upper}, therefore we can still use it to obtain
\begin{equation*}
    \|e^{\ui t|\partial_x|^s}w_\tau(x)\|_{\Gamma_{q_\tau}^{T-\tau}}^2\le C(T)\|w_\tau(x)\|^2_{L^2(\T)}\le \|V\|_\infty \|u_0\|^2_{L^2(\T)}
\end{equation*}
where we used the definition $w_\tau(x)=V(x)e^{\ui \tau (|\partial_x|^s-V)}u_0$. The above two inequalities give the upper bound of the second term on the right-hand side of \eqref{eq:integral-along-curve}
\begin{equation}\label{eq:upperbound-second}
    \biggl\|\int_0^\infty\mathbf{1}_{\tau<t} w_\tau (t-\tau,x)\d \tau\biggr\|^2_{\Gamma_p^{T}}\le TC(T)\|V\|_\infty \|u_0\|^2_{L^2(\T)}.
\end{equation}
Substituting \eqref{eq:upperbound-first} and \eqref{eq:upperbound-second} into \eqref{eq:integral-along-curve} we obtain the first claim for the upper bound.

  \begin{figure}[ht]
      \centering
\begin{tikzpicture}
  \begin{axis} [axis lines=center,ticks=none,xmin=0,xmax=5.5,ymin=0,ymax=31,
      width=0.8\linewidth,   
      height=0.7\linewidth, 
      clip=false, 
      xlabel={$t$},
      ylabel={$x$},
      ]
     \addplot [domain=0:4.5, smooth, thick] {x^3/3};
     \addplot [domain=3:4, smooth, ultra thick, red] {x^3/3};
     \addplot [domain=0:1.5, smooth, thick] {(x+3)^3/3};
     \addplot [domain=0:1, smooth, ultra thick, blue] {(x+3)^3/3};
     \draw (axis cs:0,18) node[right] {$\Gamma_{q_\tau}^{0,T-\tau}$};
     \draw (axis cs:3,18) node[right] {$\Gamma_{p}^{\tau, T}$};
     \draw (axis cs:4.35,25) node[right] {$p(t)$};
     \draw (axis cs:1.35,25) node[right] {$q_\tau(t)$};
     \draw[dashed] (axis cs:1,64/3) -- (axis cs:1,0) node[below] {$T-\tau$}; 
     \draw[dashed] (axis cs:3,27/3) -- (axis cs:3,0) node[below] {$\tau$};
     \draw[dashed] (axis cs:4,64/3) -- (axis cs:4,0) node[below] {$T$};
     \draw[dashed] (axis cs:4,64/3) -- (axis cs:0,64/3) node[left] {$p(T)$} ;
     \draw[dashed] (axis cs:3,27/3) -- (axis cs:0,27/3) node[left] {$p(\tau)$};
  \end{axis}
  \end{tikzpicture}
      \caption{The curves $\Gamma_{q_\tau}^{0,T-\tau}$ and $\Gamma_p^{\tau,T}$.}
      \label{fig:parallel-translation}
  \end{figure}

\smallskip
For the second claim, we apply Duhamel's formula \eqref{eq:duhamel} and integrate along the curve $\Gamma_p^{T}$ to obtain
\begin{equation}\label{eq:integral-along-curve-2}
    \|u\|^2_{\Gamma_p^{T}}\ge  \frac{1}{2}\|e^{\ui t|\partial_x|^s}u_0 \|_{\Gamma_p^{T}}^2
    - \biggl\|\int_0^\infty\mathbf{1}_{\tau<t} w_\tau (t-\tau,x)\d \tau\biggr\|^2_{\Gamma_p^{T}}.
\end{equation}
For the first term in the right-hand side of \eqref{eq:integral-along-curve-2}, we apply \eqref{eq:obs-curve-schroedinger} and obtain for any $T>T(p)$
\begin{equation}\label{eq:lowerbound-first}
    \|e^{\ui t|\partial_x|^s}u_0 \|_{\Gamma_p^{T}}^2\ge T\title{C}(p)\|u_0\|^2_{L^2(\T)}
\end{equation}
For the second term in the right-hand side of \eqref{eq:integral-along-curve-2}, we still have the estimate given
by \eqref{eq:upperbound-second}. Combining \eqref{eq:lowerbound-first} with \eqref{eq:upperbound-second} by 
choosing $\|V\|_\infty$ sufficiently small, we complete the proof of the second claim.
\end{proof}

\section*{Acknowledgements}

The authors wish to thank Nicolas Burq for a valuable conversation
that lead us to section~\ref{Sec:schroedinger-potential}.

\par\vspace{6pt}

B. Haak, Ph. Jaming and Y. Wang were supported by the French National
Research Agency (ANR) under contract numer ANR-24-CE40-5470.

\par\vspace{6pt}

M. Wang was supported by the National Natural Science Foundation of
China under grant 12571260.

\par\vspace{6pt}

An earlier version of this paper was posted on arXiv in March 2026 under the title ‘Curved Ingham inequalities and observability of the toroidal Schrödinger equation’ (arXiv:2603.15193). We correct some small errors in this version found by gpt.

\appendix

\section{Quick overview of observability properties for Schrödinger
  equations}\label{sec:bib-table}

Observability of Schrödinger equations on flat
tori is well understood from open sets.
Some key contributions are listed in the following table.

\begin{table}[H] 
\centering
\renewcommand{\arraystretch}{1.4} 
\setlength{\tabcolsep}{3pt} 
\begin{tabular}{p{0.42\textwidth} p{0.12\textwidth}  p{0.18\textwidth} p{0.17\textwidth}}
\toprule
Authors & dimension& potential $V$ & $\omega$  \\
\midrule
Haraux \cite{Haraux}, Jaffard \cite{Jaf1988,Jaffard-1990} & $d=2$ & $V=0$ & $\omega$ open \\
\addlinespace
Komornik \cite{Kom}, Maci\`a \cite{Macia} & $d\geq 1$ & $V=0$ & $\omega$ open \\
\addlinespace
Burq--Zworski \cite{BurqZworski} & $d=1,2$ & $V$ smooth & $\omega$ open \\
\addlinespace
Anantharaman--Maci\`a \cite{AnanthMacia} & $d\geq 1$ & $V$ bounded and continuous a.e. & $\omega$ open \\
Bourgain--Burq--Zworski \cite{BourgainBurqZworski} & $d=1,2$ & $V\in L^2$ & $\omega$ open \\
\addlinespace
Bourgain \cite{Bourgain} & $d=3$ & $V$ bounded & $\omega$ open \\
\addlinespace
Burq--Zworski \cite{BurqZworski19} & $d=1,2$ & $V=0$ & $\omega=[0,T]\times\tilde\omega$, $|\tilde\omega|>0$ \\
\addlinespace
Burq--Zhu \cite{BurqZhu1} & $d=1$ & $V\in L^2$ & $\omega\subset(0,T)\times\T$, $|\omega|>0$ \\
\addlinespace
Burq--Zhu \cite{BurqZhu1} & $d=2$ & $V\in L^q$, $q>2$ & $\omega=J\times\tilde{\omega}$, $|J|,|\tilde{\omega}|>0$ \\
\bottomrule
\end{tabular}
\end{table}

While Haraux, Jaffard and Komornik used Kahane’s theory of
lacunary series \cite{Kahane}, or, more precisely, Ingham's Inequality
(see above), other authors use microlocal techniques.

\section{Proof of Remark \ref{rem:few_points}}\label{sec:few-points}

\begin{lemma}
    Let $\eps_1,\eps_2$ be the standard basis of $\R^2$.
Let $c_{-1},c_0,c_1\in\C$ and define
$F:\R^2 \to \C$ by
\[
F(t,x)=c_{-1}e^{\ui(t-x)}+c_0+c_1e^{\ui(t+x)}.
\]
Let $(A_j)_{1\leq j\leq n}$, $n\geq 3$, be points in $\R^2$ such that
\begin{enumerate}
    \item The numbers $(\scal{A_j,\eps_2})_{1\le j\le n}$ occupy at least three distinct residue
classes modulo $\pi$;
    \item $(\scal{A_j,\eps_1+\eps_2})_{1\leq j\leq n}$ do not belong to an arithmetic progression of step $2\pi$;
     \item $(\scal{A_j,\eps_1-\eps_2})_{1\leq j\leq n}$ do not belong to an arithmetic progression of step $2\pi$.
\end{enumerate}
Assume that $F$ vanishes at the points $A_j$, $F(A_j)=0$ for $j\geq 1$,
then $c_0=c_1=c_{-1}=0$.
\end{lemma}

Note that any continuous curve that is neither an horizontal line nor one of the
diagonals will contain points that satisfy this property.

\begin{proof}
Write $A_j=(t_j,x_j)$, then we have
\begin{equation}
    \label{eq:5pointsind}
c_{-1}e^{\ui(t_j-x_j)}+c_0+c_1e^{\ui(t_j+x_j)}=0,\quad j\geq 1
\end{equation}
Assume first that $c_0=0$.  Then \eqref{eq:5pointsind} reads
$c_1e^{\ui(t_j+x_j)}+c_{-1}e^{\ui(t_j-x_j)}=0$ and this implies that
$c_{-1}=-c_1e^{2\ui x_j}$.  But then either $c_{-1}=c_1=0$ or
$e^{2\ui x_j}=e^{2\ui x_1}$ for $j\geq 2$.  But the latter requires
that $x_j=x_1+k_j\pi$, that is $\scal{A_j,\eps_2}$ belongs to an
arithmetic progression of step $\pi$.
    
\proofskip Assume now towards a contradiction that $c_0\not=0$. Then,
up to replacing $c_{\pm 1}$ by $c_{\pm1}/c_0$, we may assume that
$c_0=1$ so that \eqref{eq:5pointsind} reads
\begin{equation}  \label{eq:philippe1}
c_{-1}e^{-\ui x_j}+c_1e^{\ui x_j}=-e^{-\ui t_j}.
\end{equation}
Taking the squared modulus, this reduces to
\[
|c_{-1}|^2+|c_1|^2+2\Re(c_1\overline{c_{-1}}e^{2\ui x_j})=1.
\]
If one of $c_{\pm1}=0$, say $c_{-1}=0$ then $|c_1|=1$. We thus write
$c_1=e^{\ui\ffi}$ and \eqref{eq:philippe1} becomes
\[
e^{\ui(t_j+x_j+\ffi+\pi)}=1.
\]
But then $\scal{A_j,\eps_1+\eps_2}$ belongs to an arithmetic
progression of step $2\pi$.  If $c_1=0$, then $c_{-1}=e^{\ui\ffi}$ and
$e^{\ui(t_j-x_j+\ffi+\pi)}=1$ so that $\scal{A_j, \eps_1 - \eps_2}$
belongs to an arithmetic progression of step $2\pi$.

\proofskip It remains to exclude the case $c_1\overline{c_{-1}}\not=0$
by a contradiction argument. We now write
$c_1\overline{c_{-1}}=|c_1||c_{-1}|e^{-\ui\ffi}$ so that now
\[
|c_{-1}|^2+|c_1|^2+2|c_1||c_{-1}|\cos(2x_j-\ffi)=1
\quad\Leftrightarrow\quad
\cos(2x_j-\ffi)=\dfrac{1-|c_{-1}|^2-|c_1|^2}{2|c_1||c_{-1}|} =: \kappa.
\]
Note that necessarily $\abs{\kappa}\leq 1$ and then
\[
2x_j = \ffi \pm \arccos\kappa\quad\mod(2 \pi).
\]
Thus the $x_j$ occupy at most two residue classes modulo $\pi$,
contrary to condition (1).
\end{proof}



\end{document}